\documentclass[a4paper,12pt]{amsproc}

\usepackage{amsmath,amsthm,amssymb,mathtools,mathrsfs}
\usepackage{stmaryrd}
\usepackage{ascmac}
\usepackage{comment}
\usepackage{bm}
\usepackage{booktabs}
\usepackage{graphicx}
\usepackage{subcaption}

\usepackage{algorithm}
\usepackage{algpseudocode}

\usepackage{hyperref}

\usepackage{colortbl}
\usepackage{graphicx}
\usepackage{here}
\usepackage{time}
\usepackage[abbrev]{amsrefs}

\usepackage{xcolor}
\usepackage[capitalize,nameinlink,noabbrev,nosort]{cleveref}
\hypersetup{
    colorlinks=true,       
    linkcolor=brown,          
    citecolor=brown,        
    filecolor=brown,      
    urlcolor=brown,           
}

\usepackage{fullpage}

\makeatletter
\@namedef{subjclassname@2020}{%
  \textup{2020} Mathematics Subject Classification}
\makeatother

\newtheorem{theoremcounter}{Theorem Counter}[section]

\theoremstyle{definition}
\newtheorem{definition}[theoremcounter]{Definition}
\newtheorem{remark}[theoremcounter]{Remark}
\newtheorem{example}[theoremcounter]{Example}

\theoremstyle{plain}
\newtheorem{lemma}[theoremcounter]{Lemma}
\newtheorem{proposition}[theoremcounter]{Proposition}
\newtheorem{corollary}[theoremcounter]{Corollary}
\newtheorem{conjecture}[theoremcounter]{Conjecture}
\newtheorem{theorem}[theoremcounter]{Theorem}
\newtheorem{question}[theoremcounter]{Question}

\numberwithin{equation}{section}

\newcommand{\Z}{\mathbb{Z}}

\newcommand{\klgobel}{$(k,l)$-G\"{o}bel }

\allowdisplaybreaks

\begin{document}

\title[]{A note on non-integrality of the $(k,l)$-G\"{o}bel sequences} 

\author{Yuh Kobayashi}
\address{Department of Mathematical Sciences, Aoyama Gakuin University, 5-10-1 Fuchinobe, Chuo-ku, Sagamihara, Kanagawa, 252-5258, Japan}
\email{kobayashi@math.aoyama.ac.jp}

\author{Shin-ichiro Seki}
\address{Department of Mathematical Sciences, Aoyama Gakuin University, 5-10-1 Fuchinobe, Chuo-ku, Sagamihara, Kanagawa, 252-5258, Japan}
\email{seki@math.aoyama.ac.jp}

\thanks{This research was supported by JSPS KAKENHI Grant Numbers JP22K13960 (Kobayashi) and JP21K13762 (Seki).}

\begin{abstract}
The $(k,l)$-G\"{o}bel sequences defined by Ibstedt remain integers for the first (in some cases, many) terms, but for selected values of $(k,l)$, computations show that the terms eventually stop being integers.
It is still unresolved whether the integrality of these sequences breaks down for all $k, l\geq 2$.
In this article, we prove the non-integrality for a specific class of $(k,l)$ values.
Our proof is based on geometric arguments related to the distribution of quadratic residues modulo a prime.
\end{abstract}

\maketitle

\section{Introduction}\label{sec:Introduction}
\emph{G\"{o}bel's sequence} $(g_n)_{n\geq 0}$ is defined by the following recurrence relation:
\[
g_n=\frac{2+g_1^2+\cdots+g_{n-1}^2}{n}
\]
with $g_1=2$ (\cite[A003504]{Sloane}).
This sequence remains an integer up to $g_{42}$, but in 1975, Lenstra found that $g_{43}$ is no longer an integer.
Since then, there has been little research on G\"{o}bel's sequence.
However, after G\"{o}bel's sequence was featured in the Japanese manga \emph{Seisu-tan} \cite{KobayashiSeki2023} written by the second author of this article and Doom Kobayashi in 2023, it inspired two recent articles \cite{MatsuhiraMatsusakaTsuchida2024} and \cite{GimaEtal2024+}.
For the history of G\"{o}bel's sequence, refer to \cite{MatsuhiraMatsusakaTsuchida2024}.

In \cite{Ibstedt1990}, Ibstedt generalized G\"{o}bel's sequence to one with two parameters, $k$ and $l$. For integers $k\geq 2$ and $l\geq 0$, the sequence $(g_{k,l}(n))_{n\geq 1}$, which is called the \emph{$(k,l)$-G\"obel sequence}, is defined by
\[
g_{k,l}(n)=\frac{1}{n}\left(l+\sum_{i=1}^{n-1}g_{k,l}(i)^k\right).
\]
We can check that $g_{2,2}(n)=g_n$ holds for all $n$ and that the recurrence relation for $(k,l)$-G\"obel sequence can be rewritten as
\begin{equation}\label{eq:rec_of_(k,l)-Goebel}
(n+1)g_{k,l}(n+1)=g_{k,l}(n)\cdot\bigl(n+g_{k,l}(n)^{k-1}\bigr)
\end{equation}
with the initial condition $g_{k,l}(1)=l$.
In particular, the sequence $(g_{k,2}(n))_{n\geq 1}$ is referred to as the \emph{$k$-G\"{o}bel sequence}.

We are interested in whether the $(k,l)$-G\"{o}bel sequences stop being an integer, as in the case of G\"{o}bel's original sequence.
We define
\[
N_{k,l}\coloneqq\inf\{n\in\Z_{\geq 1} \mid g_{k,l}(n)\not\in\Z\},
\]
where $N_{k,l}=\infty$ if $g_{k,l}(n)$ remains an integer for all $n$.
Also, let $N_k\coloneqq N_{k,2}$ (\cite[A108394]{Sloane}).
We are particularly interested in the behavior of sequences $(N_{k,l})_{k,l\geq 2}$ or $(N_k)_{k\geq 2}$.
We note that the cases $l=0$ and $l=1$ are included in the definition for convenience, where we have $g_{k,0}(n)=0$ and $g_{k,1}(n)=1$ for all $n$, and in particular, $N_{k,0}=N_{k,1}=\infty$.

In \cite{MatsuhiraMatsusakaTsuchida2024}, Matsuhira, Matsusaka, and Tsuchida proved that
\[
\min_{k\geq 2}N_k=19
\]
and $N_k=19$ if and only if $k\equiv 6, 14 \pmod{18}$.
In \cite{GimaEtal2024+}, Gima, Matsusaka, Miyazaki, and Yara proved
\[
\min_{k,l\geq 2}N_{k,l}=7
\]
and $N_{k,l}=7$ if and only if $k\equiv 2 \pmod{6}$ and $l\equiv 3 \pmod{7}$.

Whereas the minimum values have been determined, the fundamental problem of the finiteness of $N_k$ for all $k\geq 2$ remains unresolved.
In fact, no approach independent of computer searches has been found so far. 
By considering the general initial value case, the following weaker problem can be formulated:
\begin{quote}
\emph{For any $k\geq 2$, does there exist some $l\geq 2$ such that $N_{k,l}$ is finite?}
\end{quote}
We partially solve this problem for $k$ belonging to an infinite class of even integers that satisfy certain conditions (see \cref{cor:main_application}).

As we are interested in the non-integrality of \klgobel sequences, we here recall the method of proving it.
Since $g_{43}\approx 5.4\times 10^{178485291567}$, fully calculating $g_{43}$ with a computer is highly impractical.
(See \cite[Theorem~4]{GimaEtal2024+} for the asymptotic expansion formula of general $(k,l)$-G\"{o}bel sequences.)
However, by using the recurrence relation and iteratively calculating the congruences modulo $43$ in $\Z_{(43)}$, we find that $43g_{43}\equiv 24 \pmod{43}$, which implies that $g_{43}\not\in\Z$.
Here, for a prime number $p$, $\Z_{(p)}$ is the localization of $\Z$ at $(p)$.
Similarly, if, based on such iterated calculations, we find a prime number $p$ such that $pg_{k,l}(p)\not\equiv0\pmod{p}$ in $\Z_{(p)}$, we can conclude the non-integrality of the $(k,l)$-G\"{o}bel sequence (and more strongly, that $N_{k,l}\leq p$).

In \cite[Section~4]{GimaEtal2024+}, the authors posed the following question:
\begin{quote}
``\emph{for any pair of integers $k$, $l\geq 2$, does there exist (infinitely many) $p \in \mathcal{P}$ such that $g_{k,l}(p) \not\in\Z_{(p)}$.}''
\end{quote}
Here, their $\mathcal{P}$ is the set of all prime numbers.
The answer to this question remains unknown, and whether such primes exist --- and if so, which ones --- cannot be determined until explicit computations are made.
Another question can also be posed by switching the roles of $(k,l)$ and $p$ as follows:
\begin{quote}
\emph{For a given prime $p$, which $k$, $l\geq 2$ satisfy $g_{k,l}(p) \not\in\Z_{(p)}$?}
\end{quote}
Because it depends only on the residue class $pg_{k,l}(p) \bmod{p}$, by Fermat's little theorem, it is sufficient to consider only the cases where $0\leq k\leq p-2$ and $0\leq l\leq p-1$.
Thus, for each given $p$, the values of $k$ and $l$ that satisfy $g_{k,l}(p) \not\in\Z_{(p)}$ can be completely determined through calculations using a computer.

While performing this calculation, we discovered a general phenomenon: for primes $p$ satisfying $p \equiv 1 \pmod{4}$, when $k=(p-1)/2$, even values of $l$ for which $g_{k,l}(p) \not\in \Z_{(p)}$ appear consecutively (see \cref{fig:blue_dots}).
\begin{figure}
\centering
\begin{subfigure}{0.32\textwidth}
    \includegraphics[height=\textwidth]{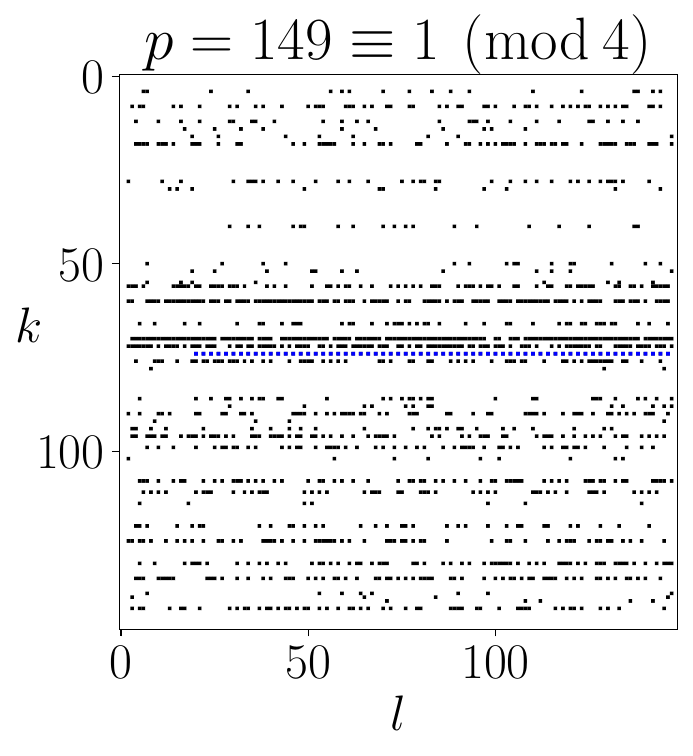}
\end{subfigure}
\hfill
\begin{subfigure}{0.32\textwidth}
    \includegraphics[height=\textwidth]{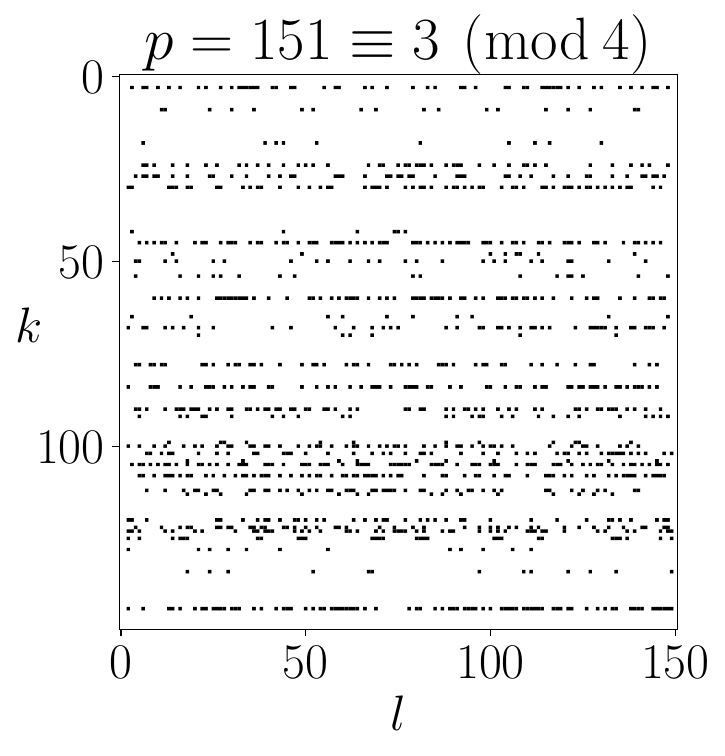}
\end{subfigure}
\hfill
\begin{subfigure}{0.32\textwidth}
    \includegraphics[height=\textwidth]{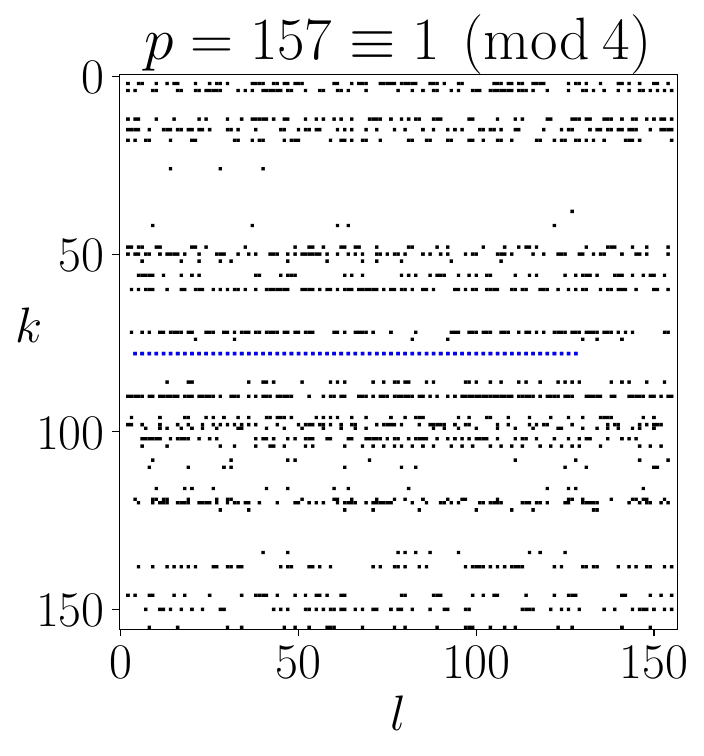}
\end{subfigure}
        
\caption{Blue ($k=(p-1)/2$) and black ($k\neq(p-1)/2$) dots represent the pair $(k,l)$ for which $g_{k,l}(p) \notin \mathbb{Z}_{(p)}$.}
\label{fig:blue_dots}
\end{figure}
In this article, we prove that this phenomenon indeed holds true.
\begin{theorem}\label{thm:main3}
Let $p$ be an odd prime number.
\begin{enumerate}
\item Assume that $p\equiv 3\pmod{4}$.
For any integer $l$ satisfying $0\leq l < p$, we have
\[
g_{\frac{p-1}{2},l}(p)\in\Z_{(p)}.
\]
\item Asuume that $p\equiv 1\pmod{4}$.
For any odd integer $l$ satisfying $0\leq l < p$, we have
\[
g_{\frac{p-1}{2},l}(p)\in\Z_{(p)}.
\]
Furthermore, if $p\geq 13$, there exists a unique decomposition
\[
\{0,2,4,\dots, p-1\}=I_p^{\mathrm{L}}\sqcup J_p\sqcup I_p^{\mathrm{R}}
\]
such that for even integers $l_{\mathrm{L}}$ and $l_{\mathrm{R}}$ with $2 \leq l_{\mathrm{L}} < l_{\mathrm{R}} \leq p-1$, the following holds$:$
\begin{align*}
I_p^{\mathrm{L}} &= \{l \in 2\Z \mid 0 \leq l < l_{\mathrm{L}}\}, \\
J_p &= \{l \in 2\Z \mid l_{\mathrm{L}} \leq l < l_{\mathrm{R}}\}, \\
I_p^{\mathrm{R}} &= \{l \in 2\Z \mid l_{\mathrm{R}} \leq l \leq p-1\}
\end{align*}
and for any even integer $l$ satisfying $0 \leq l < p$, we have
\[
g_{\frac{p-1}{2},l}(p)\in\Z_{(p)} \text{ if and only if } l\in I_p^{\mathrm{L}}\sqcup I_p^{\mathrm{R}}.
\]
\end{enumerate}
\end{theorem}
We prove this theorem through elementary geometric arguments based on the distribution of quadratic residues modulo $p$.

As an application of our result, we obtain the following corollary, which provides a partial answer to the question posed by Gima, Matsusaka, Miyazaki, and Yara.
\begin{corollary}\label{cor:main_application}
Let $k$ be a positive even integer.
Assume that there exists a prime $p\geq 13$ such that $p\equiv 1\pmod{4}$ and $k$ is a positive odd multiple of $(p-1)/2$.
Then there exists an even integer $l$ satisfying $2\leq l\leq p-3$ such that for any positive integer $l'$ with $l'\equiv l\pmod{p}$, we have $g_{k,l'}(p)\not\in\Z_{(p)}$, and hence $N_{k,l'}\leq p$.
\end{corollary}
\begin{proof}
Since $l_{\mathrm{L}} < l_{\mathrm{R}}$ in \cref{thm:main3}, it follows that $J_p\neq\varnothing$.
Therefore, for $l\in J_p$, we have $g_{\frac{p-1}{2},l}(p)\not\in\Z_{(p)}$.
By periodicity, we have the conclusion.
\end{proof}
We prepare a lemma regarding the strength of this corollary.
\begin{lemma}\label{lem:strongth}
The union
\[
\bigcup_{\substack{p\equiv 1\pmod{4}, \\ p\geq 13}}\frac{p-1}{2}\cdot(2\Z_{\geq 0}+1)
\]
cannot be covered by finitely many of the sets $\frac{p-1}{2}\cdot(2\Z_{\geq 0}+1)=\{\frac{p-1}{2}\cdot (2m+1) \mid m\in\Z_{\geq 0}\}$.
In the union above, $p$ varies over primes satisfying $p \equiv 1 \pmod{4}$ and $p\geq 13$.
\end{lemma}
\begin{proof}
Assume that there exist finitely many primes $p_1, \dots, p_t$ such that 
\[
\bigcup_{\substack{p\equiv 1\pmod{4}, \\ p\geq 13}}\frac{p-1}{2}\cdot(2\Z_{\geq 0}+1)=\bigcup_{1\leq i\leq t}\frac{p_i-1}{2}\cdot(2\Z_{\geq 0}+1)
\]
holds.
Let $s$ be the product of odd prime factors of $(p_i - 1)/2$ for all $1 \leq i \leq t$.
Take a positive odd integer $b$ such that $2b\equiv-1\pmod{s}$.
Let $q$ be a prime of the form $q=8sn+4b+1$ for some positive integer $n$.
Such a prime exists by Dirichlet's theorem on arithmetic progressions.
Since $(q-1)/2=2(2sn+b)$ is not a multiple of $(p_i-1)/2$ for any $1\leq i\leq t$, we have a contradiction.
\end{proof}
By \cref{lem:strongth}, the set of values of $k$ to which \cref{cor:main_application} can be applied is broader than the set that can be covered by finite computer calculations with previously proposed algorithms checking whether $g_{k,l}(p)\not\in\Z_{(p)}$.

We present the numerical data for $\#J_p$, $l_{\mathrm{L}}$, and $l_{\mathrm{R}}$ in \cref{subsec:comput_results:Jp_data}.
Here, for a set $A$, the notation $\#A$ denotes the cardinality of $A$.
These data (especially Figure~\ref{fig:J_p}) support the following conjecture:
\begin{conjecture}\label{conj:Num_Jp_over_p_half}
For primes $p\geq 13$ satisfying $p\equiv 1\pmod{4}$, let $J_p$ be the set defined as in $\cref{thm:main3}$.
The ratio $\#J_p/p$, which is clearly less than $1/2$ by definition, tends to $1/2$ as $p$ to infinity, that is,
\[
\lim_{p\to\infty}\frac{\#J_p}{p}=\frac{1}{2},
\]
where $p$ varies over primes satisfying $p \equiv 1 \pmod{4}$.
\end{conjecture}
This article is organized as follows.
In \cref{sec:Line_chart}, we prove a relaxed version of \cref{thm:main3} that allows the possibility of $l_{\mathrm{L}}=l_{\mathrm{R}}$ and provide a criterion for when $l_{\mathrm{L}}<l_{\mathrm{R}}$ holds (\cref{thm:empty_iff}).
In \cref{sec:special_sequences}, we investigate fundamental properties of two kinds of special sequences whose values are in $\{+1, -1\}$.
In \cref{sec:proof}, we prove $l_{\mathrm{L}}<l_{\mathrm{R}}$ by using results obtained in \cref{sec:special_sequences}.
In \cref{sec:computationla_results}, we include the data obtained from our numerical experiments, with pseudocode for the programs used in those experiments.

\section*{Acknowledgments}
The authors would like to express their sincere gratitude to Dr.~Junnosuke Koizumi, Professor Toshiki Matsusaka, Professor Stan Wagon, and an anonymous referee for their helpful comments.


\section{Geometric arguments involving Legendre symbols}\label{sec:Line_chart}
Throughout this section, we fix an odd prime $p$, and the symbol $\left(\frac{\cdot}{p}\right)$ denotes the Legendre symbol.
\begin{definition}\label{def:}
Let $l$ be an integer such that $0\leq l\leq p-1$.
We define a sequence $(\tilde{g}_l(n))_{1\leq n\leq p}$ with $\tilde{g}_l(1)=l$ and the recurrence relation:
\begin{align*}
\tilde{g}_l(n+1)\coloneqq\begin{cases}\tilde{g}_l(n)+\left(\frac{n}{p}\right)\left(\frac{\tilde{g}_l(n)}{p}\right) & \text{ if } 0<\tilde{g}_l(n)<p, \\ 0 & \text{ if } \tilde{g}_l(n)=0, \\ p & \text{ if } \tilde{g}_l(n)=p.\end{cases}
\end{align*}
\end{definition}

\begin{lemma}\label{lem:trans_to_tilde}
For any integer $n$ with $1 \leq n \leq p$, the following congruence holds in $\Z_{(p)}$$:$
\[
ng_{\frac{p-1}{2},l}(n)\equiv \tilde{g}_l(n)\pmod{p}.
\]
In particular, $g_{\frac{p-1}{2},l}(p)\in\Z_{(p)}$ if and only if $\tilde{g}_l(p)=0$ or $p$.
\end{lemma}
\begin{proof}
By the recurrence relation \eqref{eq:rec_of_(k,l)-Goebel}, we see that if $g_{\frac{p-1}{2},l}(n)\equiv 0\pmod{p}$, then $(n+1)g_{\frac{p-1}{2},l}(n+1)\equiv 0\pmod{p}$ also holds.
We assume that $ng_{\frac{p-1}{2},l}(n)$ is prime to $p$.
Since
\[
(n+1)g_{\frac{p-1}{2},l}(n+1)=ng_{\frac{p-1}{2},l}(n)+n^{-\frac{p-1}{2}}(ng_{\frac{p-1}{2},l}(n))^{\frac{p-1}{2}}
\]
holds by \eqref{eq:rec_of_(k,l)-Goebel}, we have
\[
(n+1)g_{\frac{p-1}{2},l}(n+1)\equiv ng_{\frac{p-1}{2},l}(n)+\left(\frac{n}{p}\right)^{-1}\left(\frac{ng_{\frac{p-1}{2},l}(n)}{p}\right) \pmod{p}
\]
by Euler's criterion.
Hence, we obtain the conclusion from $\left(\frac{n}{p}\right)^{-1}=\left(\frac{n}{p}\right)$.
\end{proof}

\begin{definition}\label{def:}
We define a set $\mathcal{L}$ of point sets as
\[
\mathcal{L}\coloneqq\{\{(n,a(n))\}_{1\leq n\leq p} \mid a(n)\in\Z, 0\leq a(n)\leq p \text{ for } 1\leq n\leq p\}.
\]
Here, $\{(n,a(n))\}_{1\leq n\leq p}$ is an abbreviation of the set $\{(n,a(n)) \mid 1\leq n\leq p\}$.
For $\mathsf{A}=\{(n,a(n))\}_{1\leq n\leq p}$ and $\mathsf{B}=\{(n,b(n))\}_{1\leq n\leq p}$ in $\mathcal{L}$, we use the following notation:
\begin{itemize}
\item $\mathsf{A}\leq \mathsf{B} \overset{\text{def}}\Longleftrightarrow a(n)\leq b(n)$ for all $1\leq n\leq p$,
\item $\mathsf{A}\preceq\mathsf{B} \overset{\text{def}}\Longleftrightarrow$ there exists an integer $m$ such that $0 \leq m \leq p$, where $a(n) < b(n)$ holds for any $1 \leq n \leq m$, and $a(n) = b(n)$ holds for $m < n \leq p$.
\end{itemize}
We also introduce symbols $\mathsf{G}_l$ for $0\leq l\leq p-1$, $\mathsf{diag}$, and $\overline{\mathsf{diag}}$ as elements of $\mathcal{L}$, defined respectively as follows:
\[
\mathsf{G}_l\coloneqq \{(n,\tilde{g}_l(n))\}_{1\leq n\leq p},\quad \mathsf{diag}\coloneqq\{(n,n)\}_{1\leq n\leq p},\quad \overline{\mathsf{diag}}\coloneqq\{(n,p-n)\}_{1\leq n\leq p}.
\]
\end{definition}

\begin{lemma}\label{lem:diagonal_stable}
Let $m$ be an integer such that $1\leq m\leq p$.
If $\tilde{g}_l(m)=m$, then for any $n$ such that $m\leq n\leq p$, we have $\tilde{g}_l(n)=n$.
\end{lemma}
\begin{proof}
If $\tilde{g}_l(m)=m$ for $m<p$, then
\[
\tilde{g}_l(m+1)=m+\left(\frac{m}{p}\right)^2=m+1.
\]
This process can be repeated.
\end{proof}

\begin{proposition}\label{prop:diag}
For any odd integer $l$ such that $1\leq l<p$, we have $\mathsf{diag}\preceq \mathsf{G}_l$.
\end{proposition}
\begin{proof}
First, we claim that $\mathsf{diag}\leq\mathsf{G}_l$.
If $\tilde{g}_l(n)=p$, then clearly $n\leq \tilde{g}_l(n)$.
Therefore, we may consider the case where $\tilde{g}_l(n)<p$.
At the initial point, $\tilde{g}_l(1)-1=l-1$ is a non-negative even integer, and we have either $\tilde{g}_l(n+1)-(n+1)=\tilde{g}_l(n)-n$ or $\tilde{g}_l(n+1)-(n+1)=(\tilde{g}_l(n)-n)-2$, since the values of the Legendre symbols are either $+1$ or $-1$.
Therefore, if $\tilde{g}_l(n)-n<0$ for some $n$, then there exists an $m<n$ such that $\tilde{g}_l(m)-m=0$.
However, if $\tilde{g}_l(m)-m=0$ for $m\leq p-1$, then by \cref{lem:diagonal_stable}, we have $\tilde{g}_l(n)-n=0$, which leads to a contradiction.
This proves the claim made in the beginning, and since $\tilde{g}_l(p)=p$, we can take 
$m$ such that $1\leq m\leq p$, $\tilde{g}_l(1)>1, \dots, \tilde{g}_l(m-1)>m-1$, and $\tilde{g}_l(m)=m$.
Then, by \cref{lem:diagonal_stable}, we have $\tilde{g}_l(m)-m=\tilde{g}_l(m+1)-(m+1)=\cdots=\tilde{g}_l(p)-p=0$.
This proves $\mathsf{diag}\preceq\mathsf{G}_l$.
\end{proof}

\begin{figure}
\centering
\begin{subfigure}{0.38\textwidth}
    \includegraphics[height=\textwidth]{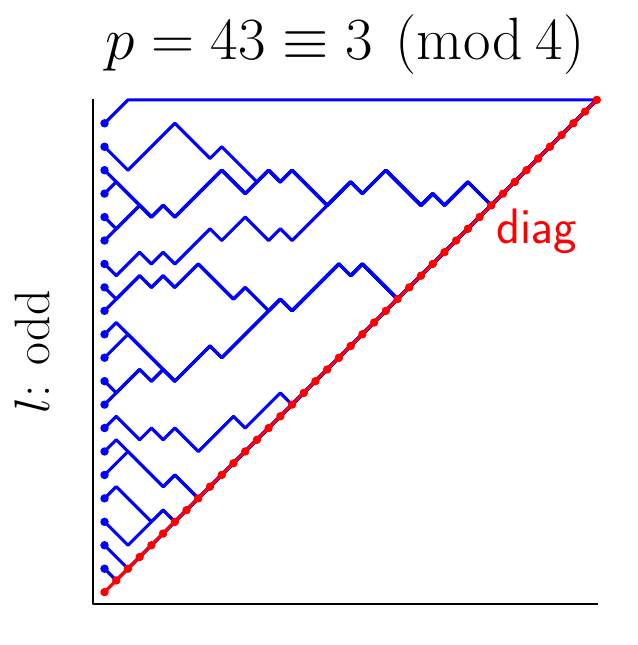}
\end{subfigure}
\hspace{0.05\textwidth} 
\begin{subfigure}{0.38\textwidth}
    \includegraphics[height=\textwidth]{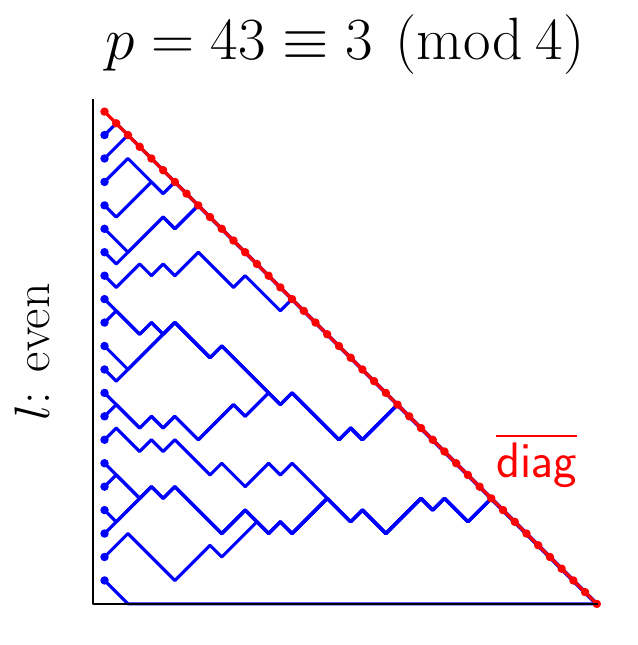}
\end{subfigure}
\caption{Plot of points belonging to sets $\mathsf{G}_l$, $\mathsf{diag}$, and $\overline{\mathsf{diag}}$.}
\label{fig:diagonals}
\end{figure}

\begin{proposition}\label{prop:anti_diag}
Assume that $p\equiv 3 \pmod{4}$.
For any even integer $l$ such that $0\leq l<p$, we have $\mathsf{G}_l\preceq\overline{\mathsf{diag}}$.
\end{proposition}
\begin{proof}
If $\tilde{g}_l(m)=p-m$ for $1\leq m<p$, then
\[
\tilde{g}_l(m+1)=p-m+\left(\frac{m}{p}\right)\left(\frac{p-m}{p}\right)=p-m+\left(\frac{-1}{p}\right)=p-(m+1),
\]
by the first supplement to the law of quadratic reciprocity and the assumption $p\equiv 3\pmod{4}$.
Based on this fact, we can prove that $\mathsf{G}_l\preceq\overline{\mathsf{diag}}$, in a manner similar to the proof of \cref{prop:diag}.
\end{proof}

\begin{proposition}\label{prop:even_ineq}
Assume that $p\equiv 1\pmod{4}$.
Then, we have
\[
\mathsf{G}_0\preceq \mathsf{G}_2\preceq \cdots \preceq \mathsf{G}_{p-3} \preceq \mathsf{G}_{p-1}.
\]
\end{proposition}
\begin{proof}
Let $l$ be an even integer such that $0\leq l\leq p-3$.
If $0<\tilde{g}_l(m)=\tilde{g}_{l+2}(m)<p$ for an integer $m$ such that $1\leq m<p$, then
\[
\tilde{g}_l(m+1)=\tilde{g}_l(m)+\left(\frac{m}{p}\right)\left(\frac{\tilde{g}_l(m)}{p}\right)=\tilde{g}_{l+2}(m)+\left(\frac{m}{p}\right)\left(\frac{\tilde{g}_{l+2}(m)}{p}\right)=\tilde{g}_{l+2}(m+1).
\]
If $\tilde{g}_l(m) = \tilde{g}_{l+2}(m) = 0$ or $\tilde{g}_l(m) = \tilde{g}_{l+2}(m) = p$, then we similarly have $\tilde{g}_l(m+1) = \tilde{g}_{l+2}(m+1)$.
By using this fact, we can prove that $\mathsf{G}_l\preceq \mathsf{G}_{l+2}$, in a manner similar to the proofs of \cref{prop:diag} and \cref{prop:anti_diag}.
\end{proof}

\begin{proof}[Proof of \cref{thm:main3} except for $l_{\mathrm{L}} < l_{\mathrm{R}}$]
If $l$ is an odd integer such that $1\leq l<p$, then by \cref{prop:diag}, $\tilde{g}_l(p)=p$.
If $p\equiv 3\pmod{4}$ and $l$ is an even integer such that $0\leq l<p$, then by \cref{prop:anti_diag}, $\tilde{g}_l(p)=0$.
Now, we assume that $p\equiv 1\pmod{4}$ and $l$ is an even integer such that $0\leq l<p$.
Set
\[
l_{\mathrm{L}}\coloneqq\max\{l\in 2\Z\mid \tilde{g}_l(p)=0\}+2 \ \text{ and } \ l_{\mathrm{R}}\coloneqq\min\{l\in 2\Z\mid \tilde{g}_l(p)=p\}.
\]
Using $l_{\mathrm{L}}$ and $l_{\mathrm{R}}$, we define the sets $I_p^{\mathrm{L}}$ and $I_p^{\mathrm{R}}$ as in the statement of \cref{thm:main3}.
Since
\[
\tilde{g}_0(p)\leq \tilde{g}_2(p)\leq \cdots \leq \tilde{g}_{p-3}(p)\leq \tilde{g}_{p-1}(p)
\]
holds by \cref{prop:even_ineq}, we see that $\tilde{g}_l(p)=0$ if and only if $l\in I_p^{\mathrm{L}}$, and $\tilde{g}_l(p)=p$ if and only if $l\in I_p^{\mathrm{R}}$.
By \cref{lem:trans_to_tilde}, $g_{\frac{p-1}{2},l}(p)\in\Z_{(p)}$ if and only if $l\in I_p^{\mathrm{L}}\sqcup I_p^{\mathrm{R}}$.
\end{proof}
\begin{remark}\label{rem:two_in_Jp}
Let $p$ denote a prime satisfying $p\equiv 1 \pmod{4}$.
In the above proof, $l_{\mathrm{L}}$ and $l_{\mathrm{R}}$ exist because $\tilde{g}_0(p)=0$ and $\tilde{g}_{p-1}(p)=p$.
The latter equality follows from
\[
\tilde{g}_{p-1}(2)=(p-1)+\left(\frac{1}{p}\right)\left(\frac{p-1}{p}\right)=p,
\]
since $p\equiv 1\pmod{4}$.
Let $J_p\coloneqq \{l\in 2\Z \mid l_{\mathrm{L}}\leq l < l_{\mathrm{R}}\}$ as in the statement of \cref{thm:main3}.
For small primes $p$, $2 \in I_p^{\mathrm{L}}$ holds, but primes $p$ for which $2\in J_p$ have also been found: 313, 1873, 2081, 2089, 2377, 4481, 5281, 6361, 6961, 7681, 8161, 8209, 8521, 8929, 9001, ....
Such primes may find some applications in the study of the non-integrality of $k$-G\"{o}bel sequences, but their distribution is not yet understood.
There are 502 of them below $10^6$.
\end{remark}

\begin{question}
Do infinitely many primes $p$ exist, satisfying $p \equiv 1 \pmod{4}$, such that $2\in J_p$?
\end{question}
\begin{definition}\label{def:}
For $\mathsf{A}=\{(n,a(n))\}_{1\leq n\leq p}\in\mathcal{L}$, we set
\begin{align*}
\mathsf{A}^+&\coloneqq\{(n,a(n))\in\mathsf{A} \mid a(n+1)=a(n)+1\}, \\
\mathsf{A}^-&\coloneqq\{(n,a(n))\in\mathsf{A} \mid a(n+1)=a(n)-1\}.
\end{align*}
We say that $\mathsf{A}\in\mathcal{L}$ is \emph{zigzag} if both $\mathsf{A}^+$ and $\mathsf{A}^-$ are non-empty.
\end{definition}
The following is the key criterion for determining whether $l_{\mathrm{L}}<l_{\mathrm{R}}$.

\begin{figure}
\centering
\begin{subfigure}{0.38\textwidth}
    \includegraphics[height=\textwidth]{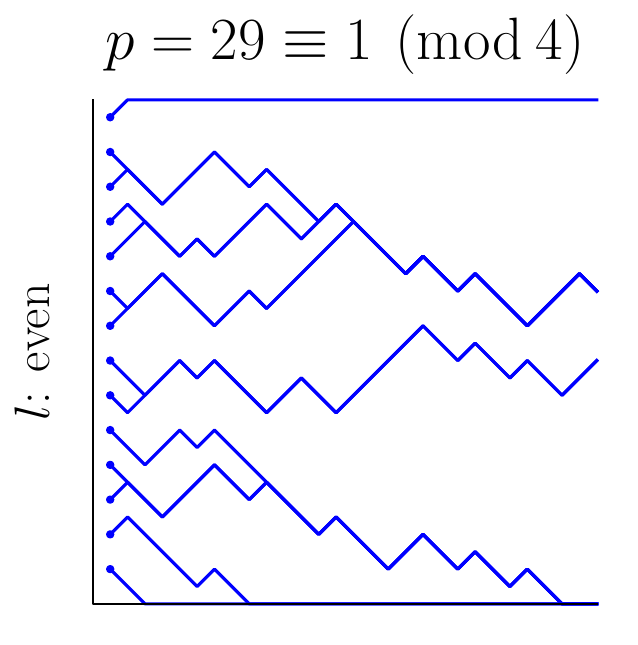}
\end{subfigure}
\hspace{0.05\textwidth} 
\begin{subfigure}{0.38\textwidth}
    \includegraphics[height=\textwidth]{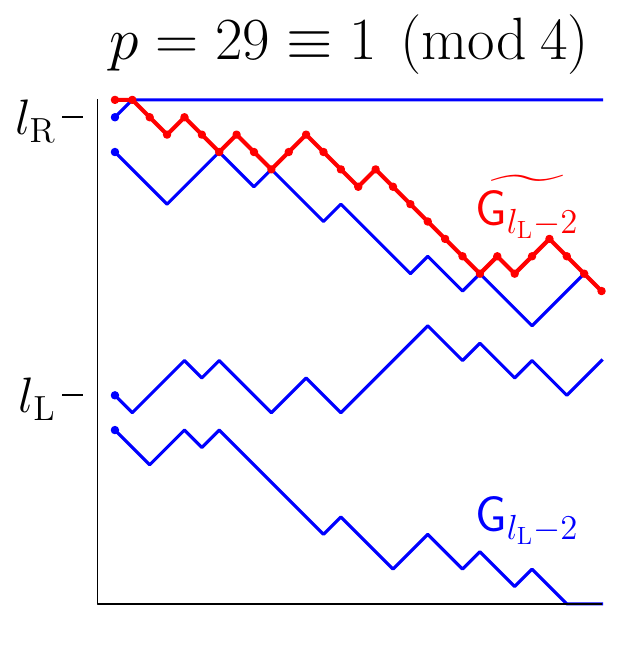}
\end{subfigure}
\caption{Plot of points belonging to sets $\mathsf{G}_l$ for even $l$ (left), and the barrier $\widetilde{\mathsf{G}_{l_{\mathrm{L}}-2}}$ (right).}
\label{fig:diagonals}
\end{figure}

\begin{theorem}\label{thm:empty_iff}
Assume that $p\equiv 1 \pmod{4}$.
The equality $l_{\mathrm{L}}=l_{\mathrm{R}}$ holds if and only if the following conditions hold for some even $l$ such that $0\leq l\leq p-3:$
\begin{enumerate}
\item\label{it:right_down} $\left(\frac{n}{p}\right)=-\left(\frac{l+1-n}{p}\right)$ holds for $1\leq n\leq l/2$,
\item\label{it:right_up} $\left(\frac{n}{p}\right)=\left(\frac{l+1+n}{p}\right)$ holds for $1\leq n\leq p-l-2$.
\end{enumerate}
\end{theorem}
\begin{proof}
We can easily check that the condition \eqref{it:right_down} is equivalent to $\mathsf{G}_l$ not being zigzag and $l\in I_p^{\mathrm{L}}$, while the condition \eqref{it:right_up} is equivalent to $\mathsf{G}_{l+2}$ not being zigzag and $l+2\in I_p^{\mathrm{R}}$.
In particular, if both conditions \eqref{it:right_down} and \eqref{it:right_up} hold, then we have $l+2=l_{\mathrm{L}}=l_{\mathrm{R}}$. 
Therefore, it is sufficient to show that if at least one of $\mathsf{G}_{l_{\mathrm{L}}-2}$ and $\mathsf{G}_{l_{\mathrm{R}}}$ is zigzag, then we have $l_{\mathrm{L}}<l_{\mathrm{R}}$.

We define $\widetilde{\mathsf{G}_{l_{\mathrm{L}}-2}}$ by
\[
\widetilde{\mathsf{G}_{l_{\mathrm{L}}-2}}\coloneqq\{(p-n,p-\tilde{g}_{l_{\mathrm{L}}-2}(n))\}_{0\leq n\leq p-1}.
\]
Here, we set $\tilde{g}_{l_{\mathrm{L}}-2}(0)\coloneqq l_{\mathrm{L}}-1$.
First, assume that $\mathsf{G}_{l_{\mathrm{L}}-2}$ is zigzag.
We prove that $\widetilde{\mathsf{G}_{l_{\mathrm{L}}-2}}$ is a \emph{barrier}, meaning that 
\begin{equation}\label{eq:barrier}
\mathsf{G}_{l_{\mathrm{L}}}\leq\widetilde{\mathsf{G}_{l_{\mathrm{L}}-2}}.
\end{equation}
Let $n^*\coloneqq\min\{n\in\Z\mid 1\leq n\leq p, \tilde{g}_{l_{\mathrm{L}}-2}(n)=0\}$.
Since $l_{\mathrm{L}}$ is even, $n^*$ is odd.
Because $\mathsf{G}_{l_{\mathrm{L}}-2}$ is zigzag, $\tilde{g}_{l_{\mathrm{L}}-2}(n)>0$ for $1\leq n\leq l_{\mathrm{L}}-1$ and $n^*\geq l_{\mathrm{L}}+1$.
Since $\tilde{g}_{l_{\mathrm{L}}-2}(p-1)=0$ holds by $p\equiv 1\pmod{4}$, we have $n^*\leq p-2$.
Therefore, for $1\leq n\leq p-n^*$,
\[
\tilde{g}_{l_{\mathrm{L}}}(n)\leq l_{\mathrm{L}}+(p-n^*-1)\leq p-2.
\]
Define $h(n)\coloneqq(p-\tilde{g}_{l_{\mathrm{L}}-2}(p-n))-\tilde{g}_{l_{\mathrm{L}}}(n)$.
By the parity of $l_{\mathrm{L}}$ and $n^*$, we see that $h(p-n^*)\geq 2$ is even.
For $p-n^*\leq n<p$, by definition, we have
\[
h(n+1)-h(n)\in\{+2, 0, -2\},
\]
unless $\tilde{g}_{l_{\mathrm{L}}}(n)=p$.
Assume that $h(n)<0$ for some $n$ such that $p-n^*+2\leq n\leq p$.
Then, we can take some $m$ and $m'$, with $3 \leq m \leq m' < p$, such that
\[
h(m-1)=+2,\quad  h(m)=\cdots =h(m')=0, \quad h(m'+1)=-2.
\]
In this situation, since $\tilde{g}_{l_{\mathrm{L}-2}}(p-m+1)=\tilde{g}_{l_{\mathrm{L}-2}}(p-m)-1$, we see that $p-m\in\mathsf{G}_{l_{\mathrm{L}}-2}^-$.
Hence,
\[
\left(\frac{p-m}{p}\right)\left(\frac{\tilde{g}_{l_{\mathrm{L}}-2}(p-m)}{p}\right)=-1
\]
must hold, and from $\tilde{g}_{l_{\mathrm{L}}}(m)=p-\tilde{g}_{l_{\mathrm{L}}-2}(p-m)$, we have
\[
\left(\frac{m}{p}\right)\left(\frac{\tilde{g}_{l_{\mathrm{L}}}(m)}{p}\right)=-1.
\]
This implies that $m\in\mathsf{G}_{l_{\mathrm{L}}}^-$.
By repeating a similar argument, we see that for any $m\leq n\leq m'$, we have $n\in\mathsf{G}_{l_{\mathrm{L}}}^-$.
In particular, $m'\in\mathsf{G}_{l_{\mathrm{L}}}^-$; however, it is impossible that $h(m'+1)=-2$.
Thus, we have $h(n)\geq 0$ for every $n$, and we get \eqref{eq:barrier}.
In particular, $\tilde{g}_{l_{\mathrm{L}}}(p)\leq p-l_{\mathrm{L}}+1<p$ and we see that $l_{\mathrm{L}}<l_{\mathrm{R}}$.

Next, assume that $\mathsf{G}_{l_{\mathrm{L}}-2}$ is not zigzag and $\mathsf{G}_{l_{\mathrm{R}}}$ is zigzag.
In this case, we have
\[
\widetilde{\mathsf{G}_{l_{\mathrm{L}}-2}}=\{(n,p) \mid 1\leq n\leq p-l_{\mathrm{L}}+1\}\cup\{(p-n,p-l_{\mathrm{L}}+n+1) \mid 0\leq n\leq l_{\mathrm{L}}-2\}.
\]
If $l_{\mathrm{L}}=l_{\mathrm{R}}$, since $\mathsf{G}_{l_{\mathrm{L}}}$ is zigzag, we can check, by a similar argument as before, that $\mathsf{G}_{l_{\mathrm{L}}}\preceq\widetilde{\mathsf{G}_{l_{\mathrm{L}}-2}}$ and $\tilde{g}_{l_{\mathrm{L}}}(p)<p$, which contradicts the definition of $l_{\mathrm{R}}$.
Therefore, we have $l_{\mathrm{L}}<l_{\mathrm{R}}$.
\end{proof} 

\section{Special sequences $(a_{p,l}(n))_{1\leq n\leq p-1}$ and $(b_{l,s}(n))_{1\leq n\leq l+1}$}\label{sec:special_sequences}
In this section, we study special sequences related to the conditions in \cref{thm:empty_iff}.
\begin{lemma}[Arithmetic billiards]\label{lem:billiards}
Let $p$ be a prime number such that $p\equiv 1\pmod{4}$, and $l$ an even integer such that $0\leq l\leq p-3$.
First, we consider the case where $l \leq (p-5)/2$.
On an $x$-$y$ plane, let $R_{p,l}$ denote the rectangle with vertices $(0,l+1)$, $(l+1,0)$, $(p/2, p/2-(l+1))$, and $(p/2-(l+1),p/2)$, and let $L_{p,l}$ denote the set of all lattice points on the boundary of $R_{p,l}$ where both coordinates are integers, excluding $(0,l+1)$ and $(l+1, 0)$.
From the vertex $(0,l+1)$, draw a path horizontally along the $x$-axis.
When it reaches the edge of $R_{p,l}$, reflect it in the direction parallel to the $y$-axis and continue drawing the path.
When it reaches the edge of $R_{p,l}$ again, reflect it in the direction parallel to the $x$-axis and continue drawing.
By repeating this process, the path passes through each element of the $p-2$ elements of $L_{p,l}$ exactly once and reaches the vertex $(l+1,0)$.
Moreover, the entire path is symmetric with respect to the line $y=x$.
For the case where $l\geq (p-1)/2$, let $R_{p,l}$ denote the rectangle with vertices $(0,p-(l+1))$, $(p-(l+1),0)$, $(p/2, (l+1)-p/2)$, and $((l+1)-p/2,p/2)$.
Then, a similar statement holds.
\end{lemma}
\begin{figure}
\centering
\includegraphics[width=0.4\textwidth]{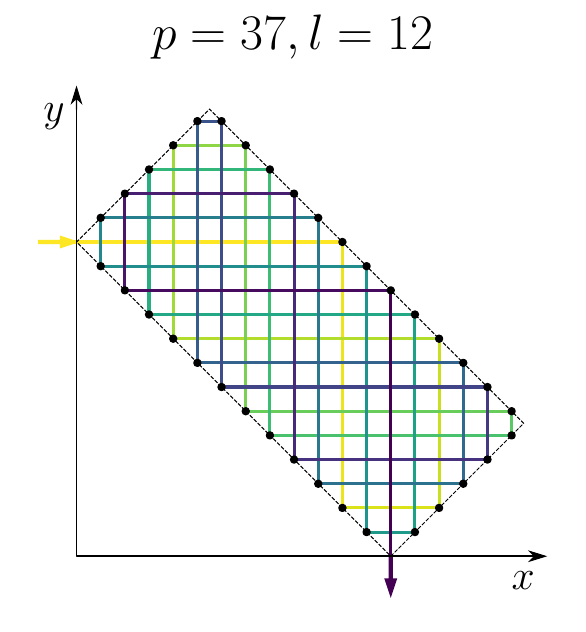}
\caption{An example of \cref{lem:billiards}.}
\label{fig:billiard}
\end{figure}
\begin{proof}
For the case where $l\leq (p-5)/2$, since $p-2(l+1)$ and $2(l+1)$ are coprime, the result can be proved by the standard argument regarding \emph{arithmetic billiards}, and thus the details are omitted.
For $l\geq (p-1)/2$, the result similarly follows because $2(l+1)-p$ and $2(p-(l+1))$ are coprime.
\end{proof}
\begin{proposition}\label{prop:unique_sequence_a}
Let $p$ be a prime number such that $p\equiv 1\pmod{4}$, and let $l$ be an even integer such that $0\leq l\leq p-3$.
Then, there exists a unique $\pm1$ sequence $(a_n)_{1\leq n\leq p-1}$ $(a_n\in\{+1,-1\})$ satisfying the following conditions$:$
\begin{enumerate}
\item $a_1=1$,
\item\label{it:seq_a2} $a_n=a_{p-n}$ holds for $1\leq n\leq (p-1)/2$,
\item\label{it:seq_a3} $a_n=-a_{l+1-n}$ holds for $1\leq n\leq l/2$,
\item\label{it:seq_a4} $a_n=a_{l+1+n}$ holds for $1\leq n\leq p-l-2$.
\end{enumerate}
\end{proposition}
\begin{proof}
First, we consider the case where $l\leq (p-5)/2$.
We define a mapping $\psi_{p,l}\colon L_{p,l}\to\{+1,-1\}$ as follows: assign $-1$ to the points in the intersection of $L_{p,l}$ and the line segment connecting $(0,l+1)$ and $(l+1,0)$, and assign $+1$ to all other points. 
Note that the condition \eqref{it:seq_a3} is equivalent to $a_na_{l+1-n}=-1$ for $1\leq n\leq l$.
The condition \eqref{it:seq_a4} is equivalent to $a_na_{l+1+n}=+1$ for $1\leq n\leq p-l-2$, and furthermore, this is equivalent to $a_na_m=+1$ for all $(n,m)\in L_{p,l}\setminus\{(1,l), (2,l-1), \dots, (l,1)\}$ under the condition \eqref{it:seq_a2}.
Therefore, it is sufficient to show that a $\pm1$ sequence $(a_n)_{1\leq n\leq(p-1)/2}$ is uniquely determined by the conditions $a_1=1$ and $\psi_{p,l}(n,m)=a_na_m$ for all $(n,m)\in L_{p,l}$.

In the setting of \cref{lem:billiards}, considering symmetry, list the first half of the lattice points in $L_{p,l}$ that the path passes through, in the order they are visited, as follows:
\begin{align*}
&(0,\sigma(1)) \longrightarrow (\tau(1),\sigma(1)) \\
&\longrightarrow (\tau(1), \sigma(2)) \longrightarrow (\tau(2), \sigma(2)) \\
&\longrightarrow (\tau(2), \sigma(3)) \longrightarrow (\tau(3), \sigma(3)) \\
&\longrightarrow \cdots \\
&\longrightarrow (\tau((p-5)/4),\sigma((p-1)/4)) \longrightarrow (\tau((p-1)/4), \sigma((p-1)/4)) \\
&\longrightarrow ((p-(l+1))/2, (p-(l+1))/2),
\end{align*}
where $\sigma$ and $\tau$ are permutations on
$\{1,2,\dots,(p-1)/2\}$ with $\sigma(1)=l+1$, $\tau(1)=p-2(l+1)$, and $\tau((p+1)/2-n)=\sigma(n)$ for $1\leq n\leq (p-1)/2$.
Since $a_{\tau(1)}a_{\sigma(1)}=\psi_{p,l}(\tau(1),\sigma(1))$, we have $a_{\tau(1)}=\psi_{p,l}(\tau(1),\sigma(1))\cdot a_{\sigma(1)}$.
Next, since $a_{\tau(1)}a_{\sigma(2)}=\psi_{p,l}(\tau(1),\sigma(2))$, we have
\[
a_{\sigma(2)}=\psi_{p,l}(\tau(1),\sigma(2))\cdot a_{\tau(1)}=\psi_{p,l}(\tau(1),\sigma(2))\psi_{p,l}(\tau(1),\sigma(1))\cdot a_{\sigma(1)}.
\]
By repeating this argument, the values of $a_{\tau(1)}, \dots, a_{\tau((p-1)/4)}$ and $a_{\sigma(2)}, \dots, a_{\sigma((p-1)/4)}$ are uniquely determined up to the value of $a_{\sigma(1)}$.
In this process, we obtain the equality $a_1=\pm a_{\sigma(1)}$ and the value of $a_{\sigma(1)}$ is also determined by the condition $a_1=1$.

For the case where $l\geq (p-1)/2$, we define a mapping $\psi_{p,l}\colon L_{p,l}\to\{+1,-1\}$ by the following: assign $+1$ to the points in the intersection of $L_{p,l}$ and the line segment connecting $(0,p-(l+1))$ and $(p-(l+1),0)$, and assign $-1$ to all other points.
Then, the rest of the argument proceeds in exactly the same manner.
\end{proof}
\begin{example}
For example, in the case of \cref{fig:billiard}, we have
\[
\sigma = \left(\begin{array}{cccccccccccccccccc}
1 & 2 & 3 & 4 & 5 & 6 & 7 & 8 & 9 & 10 & 11 & 12 & 13 & 14 & 15 & 16 & 17 & 18 \\
13 & 2 & 9 & 17 & 6 & 5 & 16 & 10 & 1 & 12 & 14 & 3 & 8 & 18 & 7 & 4 & 15 & 11 
\end{array}\right),
\]
and the values of $a_n$ up to $a_{13}$ are determined sequentially as follows:
\[
\begin{array}{cccccc}
a_{11}=a_{13}, & a_2=-a_{13}, & a_{15}=-a_{13}, & a_9=-a_{13}, & a_4=a_{13}, & a_{17}=a_{13}, \\  a_7=a_{13}, & a_6=-a_{13}, & a_{18}=-a_{13}, & a_5=-a_{13}, & a_8=a_{13}, & a_{16}=a_{13}, \\ a_3=a_{13}, & a_{10}=-a_{13}, & a_{14}=-a_{13}, & a_1=-a_{13}, & a_{12}=a_{13}.
\end{array}
\]
From $a_1=-a_{13}$, we find that $a_{13}=-1$, and $(a_{37,12}(n))_{1\leq n\leq 18}$ is given by
\[
(+1,+1,-1,-1,+1,+1,-1,-1,+1,+1,-1,-1,-1,+1,+1,-1,-1,+1).
\]
\end{example}
\begin{definition}
For each $p$ and $l$, $(a_{p,l}(n))_{1\leq n\leq p-1}$ denotes the uniquely existing sequence in \cref{prop:unique_sequence_a}.
\end{definition}
\begin{proposition}\label{prop:unique_sequence_b}
Let $l$ be a non-negative even integer and $s$ a non-negative integer such that $0\leq s\leq l$.
We assume that $2s+1$ is prime to $l+1$.
Then, there exists a unique $\pm1$ sequence $(b_n)_{1\leq n\leq l+1}$ $(b_n\in\{+1,-1\})$ satisfying the following conditions$:$
\begin{enumerate}
\item\label{it:seq_b1} $b_1=1$,
\item\label{it:seq_b2} $b_n=-b_{l+1-n}$ holds for $n\not\equiv 0\pmod{l+1}$,
\item\label{it:seq_b3} $b_{s-n}=b_{s+1+n}$ holds for all $n$.
\end{enumerate}
Here, $b_n$ is extended to any integer $n$, where if $n\equiv m\pmod{l+1}$, then $b_n=b_m$.
\end{proposition}
\begin{proof}
Take a prime number $p\geq l+3$ such that $p\equiv 1\pmod{4}$ and 
\[
\frac{p-1}{2}\equiv s\pmod{l+1}.
\]
Since $\mathrm{gcd}(2s+1,l+1)=1$, such a prime exists by Dirichlet's theorem on arithmetic progressions.
Then, we can check that the subsequence $(a_{p,l}(n))_{1\leq n\leq l+1}$ satisfies the conditions of $(b_n)_{1\leq n\leq l+1}$.
For the condition \eqref{it:seq_b3}, take integers $m$ and $t$ such that $1\leq m, t\leq l+1$, $s+1+n\equiv m\pmod{l+1}$, and $p-m\equiv t \pmod{l+1}$ hold.
Since $p-m\equiv s-n\pmod{l+1}$, we calculate
\[
b_{s+1+n}=b_m=a_{p,l}(m)=a_{p,l}(p-m)=a_{p,l}(t)=b_t=b_{s-n}.
\]
The uniqueness of $(b_n)_{1\leq n\leq l+1}$ follows from that of $(a_{p,l}(n))_{1\leq n\leq p-1}$, because the extended sequence $(b_n)_{1\leq n\leq p-1}$ satisfies all the conditions for $(a_{p,l}(n))_{1\leq n\leq p-1}$.
\end{proof}
\begin{definition}
For each $l$ and $s$, $(b_{l,s}(n))_{1\leq n\leq l+1}$ denotes the uniquely existing sequence in \cref{prop:unique_sequence_b}.
Furthermore, we extend the definition of $b_{l,s}(n)$ to any integer $n$, such that if $n\equiv m\pmod{l+1}$, then $b_{l,s}(n)=b_{l,s}(m)$.
\end{definition}
From the proof of \cref{prop:unique_sequence_b}, we see that $a_{p,l}(n)=b_{l,s}(n)$ holds for any $1\leq n\leq p-1$ when $(p-1)/2\equiv s\pmod{l+1}$.
\begin{lemma}\label{lem:sym_b1}
Let $l$ and $s$ be integers as in $\cref{prop:unique_sequence_b}$, and assume that $l\geq 2$.
Then, for all integers $n$, we have
\[
b_{l,s}(n)=\begin{cases}b_{l,l-s}(n) & \text{ if } n\not\equiv 0\pmod{l+1}, \\ -b_{l,l-s}(n) & \text{ if } n\equiv 0\pmod{l+1}.\end{cases}
\]
\end{lemma}
\begin{proof}
If both $s-n$ and $s+1+n$ are not divisible by $l+1$, then
\[
b_{l,s}(l-s-n)=-b_{l,s}(s+1+n)=-b_{l,s}(s-n)=b_{l,s}(l-s+1+n).
\]
If $s\equiv n\pmod{l+1}$, then $s+1+n$ is not divisible by $l+1$ and 
\[
b_{l,s}(l-s-n)=-b_{l,s}(s+1+n)=-b_{l,s}(s-n)=-b_{l,s}(l-s+1+n).
\]
If $s+1+n\equiv 0\pmod{l+1}$, then $s-n$ is not divisible by $l+1$ and 
\[
b_{l,s}(l-s-n)=b_{l,s}(s+1+n)=b_{l,s}(s-n)=-b_{l,s}(l-s+1+n).
\]
By these observations and the uniqueness of $(b_{l,l-s}(n))_{1\leq n\leq l+1}$, we obtain the conclusion.
\end{proof}
\begin{lemma}\label{lem:sym_b2}
Let $l$ and $s$ be integers as in $\cref{prop:unique_sequence_b}$, and assume that $l\geq 2$.
Let $n$ be an integer.
\begin{enumerate}
\item\label{it:sym_b1} If both $n$ and $n+2s+1$ are not divisible by $l+1$, then $b_{l,s}(n)=-b_{l,s}(n+2s+1)$.
\item\label{it:sym_b2} Consider the case where $s<l/2$. Set $t\coloneqq l/2-s$. If both $n$ and $n+2t$ are not divisible by $l+1$, then $b_{l,s}(n)=-b_{l,s}(n+2t)$.
\item\label{it:sym_b3} Consider the case where $s<l/2$. Set $t\coloneqq l/2-s$. If both $t-n$ and $t+n$ are not divisible by $l+1$, then $b_{l,s}(t-n)=b_{l,s}(t+n)$.
\end{enumerate}
\end{lemma}
\begin{proof}
For \eqref{it:sym_b1}, using \cref{lem:sym_b1}, we calculate
\begin{align*}
b_{l,s}(n)&=b_{l,l-s}(n)=b_{l,l-s}(l-s-(l-s-n))\\
&=b_{l,l-s}(l-s+1+(l-s-n))\\
&=b_{l,l-s}(l+1-(n+2s+1))\\
&=-b_{l,l-s}(n+2s+1)\\
&=-b_{l,s}(n+2s+1).
\end{align*}
For \eqref{it:sym_b2}, using \eqref{it:sym_b1}, we have
\[
b_{l,s}(n+2t)=-b_{l,s}(n+2t+2s+1)=-b_{l,s}(n+l+1)=-b_{l,s}(n).
\]
For \eqref{it:sym_b3}, using \eqref{it:sym_b2}, we have
\[
b_{l,s}(t-n)=-b_{l,s}(l+1+n-t)=-b_{l,s}(n-t)=b_{l,s}(n-t+2t)=b_{l,s}(n+t).\qedhere
\]
\end{proof}

\section{Proof of $l_{\mathrm{L}}<l_{\mathrm{R}}$}\label{sec:proof}
Let $p$ be a prime number such that $p\equiv 1 \pmod{4}$.
Since the Legendre symbol has the symmetry $\left(\frac{n}{p}\right)=\left(\frac{p-n}{p}\right)$, in order to prove $l_{\mathrm{L}}<l_{\mathrm{R}}$, by \cref{thm:empty_iff}, it is sufficient to show that
\[
\left(\left(\frac{n}{p}\right)\right)_{1\leq n\leq p-1} \neq (a_{p,l}(n))_{1\leq n\leq p-1} \quad (\text{as sequences})
\]
for all even integers $0\leq l\leq p-3$.

In the following, we assume that $l$ is an even integer such that $0\leq l\leq p-3$ and that $p\geq 13$.
Moreover, we assume that
\[
\left(\left(\frac{n}{p}\right)\right)_{1\leq n\leq p-1} = (a_{p,l}(n))_{1\leq n\leq p-1} \quad (\text{as sequences}).
\]
If $l=0$, then $a_{p,0}(n)=1$ for all $1\leq n\leq p-1$, which contradicts the fact that the half of non-zero residues modulo $p$ are not quadratic residues.
Therefore, we may assume that $l\geq 2$.
Let $s$ be the unique integer satisfying $(p-1)/2\equiv s\pmod{l+1}$, with $0\leq s\leq l$.
Then, under our assumption, we have
\[
\left(\frac{n}{p}\right)=b_{l,s}(n)\quad \text{ for all } 1\leq n\leq p-1.
\]
Other than in a few exceptional cases, a contradiction can be derived from the information for $1\leq n\leq l$ alone.
By \cref{lem:sym_b1}, we may assume that $s<l/2$. (Since $2\cdot(l/2)+1=l+1$, $s=l/2$ is impossible.)

\

\noindent \textbf{The case I}: Assume that $p\equiv 1\pmod{8}$ and $s<l/4$.
Then, $4s+2\leq l$.
By \eqref{it:sym_b1} of \cref{lem:sym_b2}, we have
\[
b_{l,s}(2s+1)=-b_{l,s}(4s+2).
\]
Therefore, by our assumption, we have
\[
\left(\frac{2s+1}{p}\right)=-\left(\frac{4s+2}{p}\right)=-\left(\frac{2}{p}\right)\left(\frac{2s+1}{p}\right)
\]
or $\left(\frac{2}{p}\right)=-1$, which contradicts the second supplement to tha law of quadratic reciprocity.

\

\noindent \textbf{The case II}: Assume that $p\equiv 1\pmod{8}$ and $l/4\leq s<l/2$.
Set $t\coloneqq l/2-s$.
Then, $0<2t<4t\leq l$.
By \eqref{it:sym_b2} of \cref{lem:sym_b2}, we have
\[
b_{l,s}(2t)=-b_{l,s}(4t).
\]
Therefore, by our assumption, we have
\[
\left(\frac{2t}{p}\right)=-\left(\frac{4t}{p}\right)=-\left(\frac{2}{p}\right)\left(\frac{2t}{p}\right)
\]
or $\left(\frac{2}{p}\right)=-1$.

\

\noindent \textbf{The case III}: Assume that $p\equiv 5\pmod{8}$ and $s<l/4$.
Then, $4s+2\leq l$.
If $s=0$, then we can easily check that
\[
(b_{l,0}(n))_{1\leq n\leq l+1}=(+1,-1,+1,-1,\dots,+1,-1,+1).
\]
So, if $l\geq 4$, then $b_{l,0}(4)=-1\neq \left(\frac{4}{p}\right)$, which is a contradiction.
If $l=2$, then 
\[
(b_{2,0}(n))_{1\leq n\leq p-1}=(+1,-1,+1,+1,-1,+1,\dots),
\]
in particular, $b_{2,0}(2)b_{2,0}(3)=-b_{2,0}(6)$ holds, and this contradicts $\left(\frac{2}{p}\right)\left(\frac{3}{p}\right)=\left(\frac{6}{p}\right)$.
(Recall $p\geq 13$.)
Thus, we may assume that $s>0$.
Since
\[
b_{l,s}(2s)=b_{l,s}(s+1+(s-1))=b_{l,s}(s-(s-1))=b_{l,s}(1)=1,
\]
we have
\[
b_{l,s}(s)=\left(\frac{s}{p}\right)=\left(\frac{2}{p}\right)\left(\frac{2s}{p}\right)=-b_{l,s}(2s)=-1.
\]
Hence, by \eqref{it:sym_b1} of \cref{lem:sym_b2}, we have
\begin{align*}
b_{l,s}(2s-1)&=-b_{l,s}(2s-1+(2s+1))=-b_{l,s}(4s)\\
&=-\left(\frac{4s}{p}\right)=-\left(\frac{4}{p}\right)\left(\frac{s}{p}\right)\\
&=-b_{l,s}(s)=1.
\end{align*}
On the other hand, 
\[
b_{l,s}(2s-1)=b_{l,s}(s+1+(s-2))=b_{l,s}(s-(s-2))=b_{l,s}(2)=\left(\frac{2}{p}\right)=-1.
\]
Therefore, we arrive at a contradiction.

\

\noindent \textbf{The case IV}: Assume that $p\equiv 5\pmod{8}$ and $l/4\leq s<l/2$.
Set $t\coloneqq l/2-s$.
Then, $0<t< 4t\leq l$.
If $t=1$, then we can easily check that
\begin{align*}
&(b_{l,l/2-1}(n))_{1\leq n\leq l+1}\\
&=\begin{cases} (\underline{+1,+1,-1,-1},\dots,\underline{+1,+1,-1,-1},-1) & \text{ if } l\equiv 0\pmod{4}, \\ (+1,\underline{-1,-1,+1,+1},\dots,\underline{-1,-1,+1,+1},-1,+1) & \text{ if } l\equiv 2\pmod{4}.\end{cases}
\end{align*}
Here, the underlined parts indicate the repeated patterns.
For the case where $l\equiv 0\pmod{4}$, $b_{l,l/2-1}(2)=+1$ contradicts the second supplement.
For the case where $l\equiv 2\pmod{4}$, we have $b_{l,l/2-1}(2)b_{l,l/2-1}(3)=-b_{l,l/2-1}(6)$, and this contradicts $\left(\frac{2}{p}\right)\left(\frac{3}{p}\right)=\left(\frac{6}{p}\right)$.
(Note that $l\geq 4t=4$.)
If $t=2$ and $l\geq 10$, then by \eqref{it:sym_b2} of \cref{lem:sym_b2}, we have
\[
b_{l,l/2-2}(10)=-b_{l,l/2-2}(6)=b_{l,l/2-2}(2)=\left(\frac{2}{p}\right)=-1
\]
and $b_{l,l/2-2}(5)=-b_{l,l/2-2}(1)=-1$.
These lead to $\left(\frac{2}{p}\right)\left(\frac{5}{p}\right)\neq \left(\frac{10}{p}\right)$.
We can easily check that
\[
(b_{8,2}(n))_{1\leq n\leq 9}=(+1,+1,+1,+1,-1,-1,-1,-1,-1),
\]
which leads to $\left(\frac{2}{p}\right)=+1$. (Note that $l\geq 4t=8$.)
Thus, we may assume that $t>2$.
By \eqref{it:sym_b3} of \cref{lem:sym_b2}, we have
\[
b_{l,s}(2t-2)=b_{l,s}(t-(2-t))=b_{l,s}(t+(2-t))=b_{l,s}(2)=\left(\frac{2}{p}\right)=-1
\]
and
\[
\left(\frac{t-1}{p}\right)=\left(\frac{2}{p}\right)\left(\frac{2t-2}{p}\right)=-b_{l,s}(2t-2)=1.
\]
Therefore, 
\[
b_{l,s}(4t-4)=\left(\frac{4t-4}{p}\right)=\left(\frac{4}{p}\right)\left(\frac{t-1}{p}\right)=1.
\]
On the other hand, by \eqref{it:sym_b2} and \eqref{it:sym_b3} of \cref{lem:sym_b2},
\begin{align*}
b_{l,s}(4t-4)&=b_{l,s}(2t-4+2t)=-b_{l,s}(2t-4)\\
&=-b_{l,s}(t+(t-4))=-b_{l,s}(t-(t-4))\\
&=-b_{l,s}(4)=-\left(\frac{4}{p}\right)=-1,
\end{align*}
which leads to a contradiction.

Since contradictions arose in all cases, the proof of the main theorem has been completed. Q.E.D.
\begin{remark}
In the proof for the case $l\geq 2$, the fact that the Legendre symbol is a completely multiplicative function is not fully utilized, and only its multiplicativity within a limited range is used.
Neither is the proof critically dependent on any other properties of the Legendre symbol.
In fact, the second supplement to the law of quadratic reciprocity is only used for the case analysis to determine whether $a_{p,l}(2)$ is $+1$ or $-1$.
Therefore, we have actually obtained the following theorem concerning the properties of the sequence $(a_{p,l}(n))_{1\leq n\leq p-1}$.
\end{remark}
\begin{theorem}
Let $p\geq 13$ be a prime number such that $p\equiv 1 \pmod{4}$.
Let $l$ be an even integer such that $2\leq l \leq p-3$.
Then there exists a positive integer $m$ such that $2m\leq p-3$ and $a_{p,l}(2m) \neq a_{p,l}(2)a_{p,l}(m)$.
\end{theorem}

\section{Computational results}\label{sec:computationla_results}
In this section, we first present the details of the computations regarding the $N_k$ values. After we determined the exact $N_k$ values for $k$ up to $10^7$ (\cref{thm:main1}), we further explored the finiteness of $N_k$ using a sieving algorithm to obtain the upper bound of $N_k$ for $k$ up to $10^{14}$ (\cref{thm:main2}). These algorithms can be easily generalized to the case of arbitrary $l$, and the latter algorithm for general $(k,l)$ particularly inspired the discovery of our main result.
We also present the computational results that support \cref{conj:Num_Jp_over_p_half}.

\subsection{Determining exact values and upper bounds of $N_k$}\label{subsec:comput_results:k_gobel}
We extended the calculations of $N_k$, which was previously only listed up to $k\leq 61$ in \cite[A108394]{Sloane}, to obtain the following results:
\begin{theorem}\label{thm:main1}
For all $2\leq k\leq 10^7$, $N_k$ is explicitly determined $($see \cite{Kobayashi_github}$)$, and the largest $N_k$ in this range is $N_{2600725}=9011$.
Furthermore, $N_k$ is a prime number for $8649270$ $(86.49\%)$ of $9999999$ possible values of $k\leq 10^7$.
\end{theorem}
We computed the exact values of $N_k$ according to \cref{alg:gobel_integrality}, which will be explained in detail in \cref{subsec:algorithms}.

\begin{remark}
After the first version of our article had been posted to arXiv, a short Mathematica code for calculating $N_k$ by Buck, Motley, and Wagon was made publicly available at \cite[A108394]{Sloane}.
The exact values of $N_k$ for $k$ up to $10^5$ are also available at the same site as of writing of the current version of our article.
\end{remark}
We tabulate the first 360 terms of $N_k$ in \cref{tab:gobel_plain}.
The full data of computed $N_k$ values is available at \cite{Kobayashi_github}.

\begin{table}[]
    \centering
    \scriptsize
    \begin{tabular}{r|rrrrrrrrrrrrrrrrrr}
 & 0 & 1 & 2 & 3 & 4 & 5 & 6 & 7 & 8 & 9 & 10 & 11 & 12 & 13 & 14 & 15 & 16 & 17 \\
\hline
0 & $\infty$ & $\infty$ & 43 & 89 & 97 & 214 & 19 & 239 & 37 & 79 & 83 & 239 & 31 & 431 & 19 & 79 & 23 & 827 \\
18 & 43 & 173 & 31 & 103 & 94 & 73 & 19 & 243 & 141 & 101 & 53 & 811 & 47 & 1077 & 19 & 251 & 29 & 311 \\
36 & 134 & 71 & 23 & 86 & 43 & 47 & 19 & 419 & 31 & 191 & 83 & 337 & 59 & 1559 & 19 & 127 & 109 & 163 \\
54 & 67 & 353 & 83 & 191 & 83 & 107 & 19 & 503 & 29 & 191 & 47 & 83 & 51 & 1907 & 19 & 131 & 37 & 137 \\
72 & 31 & 214 & 31 & 127 & 47 & 443 & 19 & 173 & 31 & 227 & 23 & 337 & 83 & 563 & 19 & 47 & 166 & 487 \\
90 & 29 & 89 & 83 & 79 & 137 & 73 & 19 & 2039 & 62 & 218 & 59 & 127 & 31 & 81 & 19 & 239 & 37 & 71 \\
108 & 46 & 167 & 31 & 457 & 101 & 179 & 19 & 173 & 37 & 179 & 29 & 191 & 67 & 563 & 19 & 86 & 43 & 151 \\
126 & 23 & 101 & 43 & 81 & 59 & 139 & 19 & 47 & 31 & 249 & 46 & 101 & 83 & 647 & 19 & 179 & 25 & 103 \\
144 & 43 & 486 & 29 & 83 & 23 & 167 & 19 & 167 & 37 & 331 & 53 & 167 & 47 & 167 & 19 & 25 & 59 & 326 \\
162 & 31 & 191 & 31 & 79 & 43 & 73 & 19 & 479 & 23 & 79 & 47 & 359 & 29 & 359 & 19 & 71 & 37 & 47 \\
180 & 97 & 839 & 61 & 431 & 46 & 227 & 19 & 827 & 37 & 241 & 159 & 118 & 23 & 167 & 19 & 103 & 97 & 179 \\
198 & 47 & 131 & 31 & 127 & 29 & 254 & 19 & 251 & 46 & 137 & 43 & 331 & 79 & 479 & 19 & 239 & 23 & 163 \\
216 & 47 & 214 & 47 & 347 & 83 & 307 & 19 & 251 & 31 & 47 & 173 & 101 & 43 & 83 & 19 & 229 & 173 & 751 \\
234 & 113 & 191 & 23 & 101 & 53 & 73 & 19 & 1149 & 61 & 79 & 47 & 103 & 59 & 71 & 19 & 79 & 37 & 173 \\
252 & 31 & 191 & 31 & 251 & 83 & 201 & 19 & 233 & 31 & 499 & 47 & 313 & 47 & 359 & 19 & 89 & 46 & 139 \\
270 & 43 & 47 & 46 & 151 & 59 & 151 & 19 & 863 & 25 & 223 & 23 & 614 & 31 & 191 & 19 & 163 & 29 & 173 \\
288 & 53 & 431 & 31 & 81 & 43 & 311 & 19 & 179 & 37 & 103 & 101 & 129 & 113 & 1559 & 19 & 127 & 59 & 331 \\
306 & 34 & 227 & 47 & 179 & 47 & 73 & 19 & 227 & 29 & 158 & 47 & 47 & 46 & 179 & 19 & 79 & 37 & 167 \\
324 & 23 & 491 & 109 & 79 & 141 & 131 & 19 & 479 & 37 & 86 & 43 & 193 & 47 & 101 & 19 & 223 & 47 & 129 \\
342 & 29 & 137 & 31 & 311 & 23 & 103 & 19 & 563 & 31 & 169 & 47 & 127 & 34 & 89 & 19 & 337 & 37 & 167 \\
\vdots &&&&&&$\vdots$&&&&&&&$\vdots$&&&&& \\

    \end{tabular}
    \caption{Computed values of $N_k$ for $k$-G\"{o}bel sequences with $k$ up to 359. Each entry represents the value of $N_k$ for $k$ equal to the sum of the numbers at the top and left. Note that $N_k = 19$ for $k\equiv 6, 14 \pmod{18}$.}
    \label{tab:gobel_plain}
\end{table}

It appears that the typical values of $N_k$ are larger for $k\equiv 1\pmod{2}$ than for the other case, and $N_k$ seems to get even larger for $k\equiv 1\pmod{6}$.
Here we let $\overline{N}_{a\bmod{d}}$ denote the arithmetic mean of the exact values of $N_k$ with $2\leq k\leq 10^7$ and $k\equiv a\pmod{d}$.
We found that $\overline{N}_{0\bmod{2}} \approx 47.4$, $\overline{N}_{1\bmod{2}} \approx 263.2$, and $\overline{N}_{1\bmod{6}} \approx 383.1$.
Under the condition $k\equiv 1\pmod{6}$, we found almost no correlation between the values of $N_k$ and whether $k$ is a prime number.
When $N_k$ values are grouped by whether $k$ is a prime, the arithmetic mean of $N_k$ for $k\equiv 1\pmod{6}$ with $2\leq k\leq 10^7$ is approximately 382.2 for 332194 prime values of $k$ and 383.3 for 1334472 composite values of $k$.
From these numerical results, several questions naturally arise.
First, does $\overline{N}_{1\bmod{d}} \geq \overline{N}_{a\bmod{d}}$ hold for every $a$?
We found that it is not true: for integer $d$ with $2\leq d\leq 10^3$, there are only 204 values of $d$ such that $\overline{N}_{1\bmod{d}} = \max_{0\leq a<d}\overline{N}_{a\bmod{d}}$.
However, for the remaining 795 cases, we still have $\overline{N}_{1\bmod{d}} > 0.9486\times \max_{0\leq a<d}\overline{N}_{a\bmod{d}}$. Therefore, although the rigorous answer to the question is false, $\overline{N}_{1\bmod{d}}$ is still close to the maximum in many cases.
Second, does $\overline{N}_{1\bmod{a}} < \overline{N}_{1\bmod{b}}$ hold for every pair $(a,b)$ of integers greater than 1, with $a\neq b$ and $a\mid b$?
The answer is again negative in a strict sense: among the 5070 pairs of $(a,b)$ with $a,b\leq 10^3$ and satisfying the conditions above, the inequality in question is false in 201 cases.
Nevertheless, for the 201 cases where the inequality does not hold, it is still true that $0.9824\times\overline{N}_{1\bmod{a}} < \overline{N}_{1\bmod{b}}$.
Even when we add the condition that $b/a$ should be a prime number for a pair $(a,b)$, the inequality is false in just 159 cases among 1958 possible pairs of $(a,b)$.
It can be concluded that our second question seems true in an approximate sense.
The pairs $(a,b)$ for which $\overline{N}_{1\bmod{a}} > \overline{N}_{1\bmod{b}}$ with $a\mid b$ include $(8, 16)$, $(16, 32)$, $(24, 48)$, and $(243, 729)$.
Based on the hypothesis that a greater variety of prime factors in $d$ leads to a larger $N_k$ value, we calculated $\overline{N}_{1\bmod{p\#}}$ for prime $p$ up to 13 and obtained the following:
\[
\overline{N}_{1\bmod{3\#}} < \overline{N}_{1\bmod{5\#}} < \overline{N}_{1\bmod{7\#}} < \overline{N}_{1\bmod{11\#}} < \overline{N}_{1\bmod{13\#}} \approx 600.3,
\]
where $p\#$ denotes the primorial of $p$.
Within the range $2\leq k\leq 10^7$, we have only 333 values of $k$ satisfying $k\equiv 1 \pmod{13\#}$.
We additionally calculated $N_{c(p\#)+1}$ with $p$ prime up to 29 for $1\leq c\leq 3000$ by \cref{alg:gobel_integrality}.
Letting $\overline{\overline{N}}_{1\bmod{p\#}}$ denote the arithmetic mean of these 3000 terms, we obtained the following:
\[
\overline{\overline{N}}_{1\bmod{p\#}}<\overline{\overline{N}}_{1\bmod{q\#}}
\]
for any primes $p$ and $q$ with $2\leq p<q\leq 29$, and $\overline{\overline{N}}_{1\bmod{29\#}} \approx 753.5$.
We are currently unable to explain this phenomenon theoretically.
Nevertheless, it could be a topic of future research to explore the set of $k$ for which the value of $N_k$ is likely to be large.

We found that $N_k$ is a prime number for 8649270 (86.49\%) of 9999999 possible values of $k$ in this range.
In \cite[Section~3]{MatsuhiraMatsusakaTsuchida2024}, it is stated that $\{N_k\mid k\geq 2\}=\{19, 23, 31, 37, 43, \dots\}$ based on the data in \cite[A108394]{Sloane}, and they observed that $N_{142}=25$.
We also found $N_{306}=34$ and a wide variety of composite $N_k$ values: among the 621 values of $N_k \leq 1000$ for $2\leq k\leq 10^7$, only 160 are prime, while the other 461 are composite.
We note that 99 of these 461 composite numbers occur less than 10 times in nearly $10^7$ data, and 27 of them occur only once.
It is possible that additional composite numbers less than or equal to 1000, which did not occur in our data, may start to appear as we calculate more $N_k$ values.

\begin{table}[]
    \centering
\begin{tabular}{rrl|rrl|rrl}
\toprule
$k$ & $N_k$ & & $k$ & $N_k$ & & $k$ & $N_k$ & \\
\midrule
2 & 43 & prime & 49 & 1559 & prime & 107161 & 4463 & prime \\
3 & 89 & prime & 67 & 1907 & prime & 121801 & 5507 & prime \\
4 & 97 & prime & 97 & 2039 & prime & 707197 & 5879 & prime \\
5 & 214 & composite & 1441 & 2339 & prime & 832321 & 7127 & prime \\
7 & 239 & prime & 4189 & 2589 & composite & 1412161 & 7883 & prime \\
13 & 431 & prime & 5581 & 2687 & prime & 2600725 & 9011 & prime \\
17 & 827 & prime & 8209 & 2939 & prime & & & \\
31 & 1077 & composite & 13201 & 4139 & prime & & & \\
\bottomrule
\end{tabular}
    \caption{All values of $k$ with $2\leq k\leq 10^7$ such that $N_{k^\prime} < N_k$ holds for all $k^\prime$ satisfying $2\leq {k^\prime}<k$, along with the corresponding values of $N_k$ and whether $N_k$ is a prime number.}
    \label{tab:gobel_records}
\end{table}

In \cref{tab:gobel_records}, we present the ``new records'' of $N_k$ that have appeared as $k$ increments.
From the table, we can see that the growth of the largest $N_k$ value as $k$ increases is rather slow.
Additionally, we found $N_{1984\times 29\#+1}=10007$ in the exact determination of $\overline{\overline{N}}_{1\bmod{29\#}}$.
We note that $1984\times 29\# = 2^6\times 31\#$.

If we limit ourselves to the upper bound estimation of $N_k$, we can obtain further results regarding the non-integrality of the $k$-G\"{o}bel sequences.
Once we identify the values of $k$ for a given prime $p$ that result in non-integrality of the $k$-G\"{o}bel sequences,
we can sieve out the values of $k'$ that satisfy $k'\equiv k\pmod{p-1}$ by Fermat's little theorem.
If a range of integers are all sieved out by a finite number of congruence classes with guaranteed non-integrality, we can verify that $N_k$ is bounded for that range of $k$.
The following is concluded from our computation in this approach.
\begin{theorem}\label{thm:main2}
For all $2\leq k\leq 10^{14}$, $N_k\leq 29363$.
\end{theorem}
Refer to \cref{subsec:algorithms} for an explanation of our algorithm for this result.

\subsection{Computation of $\#J_p$, $l_{\mathrm{L}}$, and $l_{\mathrm{R}}$ for varying $p\equiv 1\pmod{4}$}\label{subsec:comput_results:Jp_data}

In \cref{tab:jp_length}, we present the value of $\#J_p$, $l_\mathrm{L}$, and $l_\mathrm{R}$ as defined in \cref{thm:main3} for the first 32 prime numbers $p$ with $p\geq 13$ and $p\equiv 1\pmod{4}$.
Although $l_\mathrm{L} > 2$ for the first few dozen primes $p\equiv 1\pmod{4}$, there also exist prime numbers for which $l_\mathrm{L} = 2$, as we have noted in \cref{rem:two_in_Jp}.
In addition, we see that the general trend is that $\#J_p$ tends to increase with $p$, although it does not increase strictly monotonically.

\begin{table}[t]
    \centering \small
\begin{tabular}{rrrr|rrrr|rrrr|rrrr}
\toprule
$p$ & $\#J_p$ & $l_L$ & $l_R$ & $p$ & $\#J_p$ & $l_L$ & $l_R$ & $p$ & $\#J_p$ & $l_L$ & $l_R$ & $p$ & $\#J_p$ & $l_L$ & $l_R$ \\
\midrule
13 & 3 & 4 & 10 & 89 & 38 & 8 & 84 & 173 & 80 & 4 & 164 & 269 & 122 & 16 & 260 \\
17 & 3 & 8 & 14 & 97 & 43 & 8 & 94 & 181 & 78 & 14 & 170 & 277 & 114 & 22 & 250 \\
29 & 8 & 12 & 28 & 101 & 42 & 4 & 88 & 193 & 85 & 6 & 176 & 281 & 116 & 16 & 248 \\
37 & 12 & 10 & 34 & 109 & 45 & 12 & 102 & 197 & 91 & 14 & 196 & 293 & 127 & 30 & 284 \\
41 & 14 & 8 & 36 & 113 & 41 & 30 & 112 & 229 & 103 & 12 & 218 & 313 & 136 & 2 & 274 \\
53 & 21 & 10 & 52 & 137 & 60 & 8 & 128 & 233 & 102 & 14 & 218 & 317 & 149 & 10 & 308 \\
61 & 22 & 12 & 56 & 149 & 64 & 20 & 148 & 241 & 106 & 24 & 236 & 337 & 161 & 6 & 328 \\
73 & 27 & 16 & 70 & 157 & 63 & 4 & 130 & 257 & 119 & 8 & 246 & 349 & 162 & 20 & 344 \\
\bottomrule
\end{tabular}
    \caption{First 32 values of $\#J_p$, $l_{\mathrm{L}}$, and $l_{\mathrm{R}}$ for primes $p\equiv 1\pmod{4}$ with $p\geq 13$.}
    \label{tab:jp_length}
\end{table}

\begin{figure}[t]
\centering
\hspace{-0.15\textwidth}
\begin{subfigure}{0.38\textwidth}
    \includegraphics[height=\linewidth]{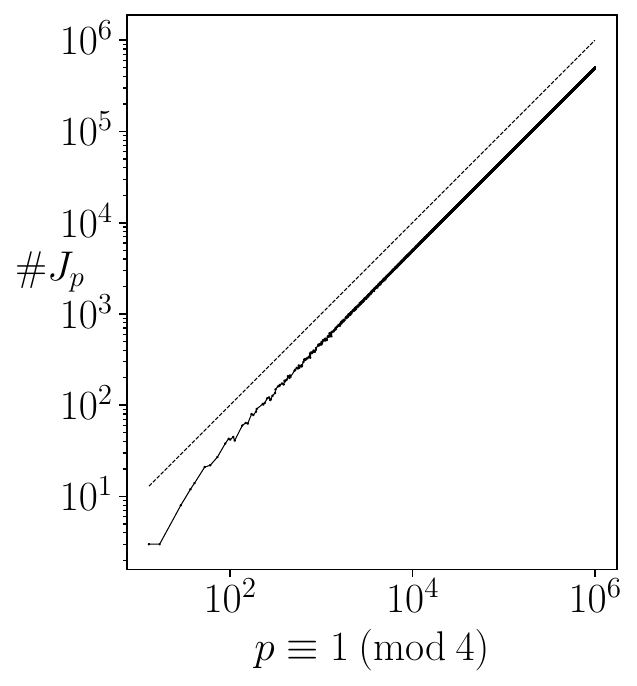}
\end{subfigure}
\begin{subfigure}{0.38\textwidth}
    \includegraphics[height=\linewidth]{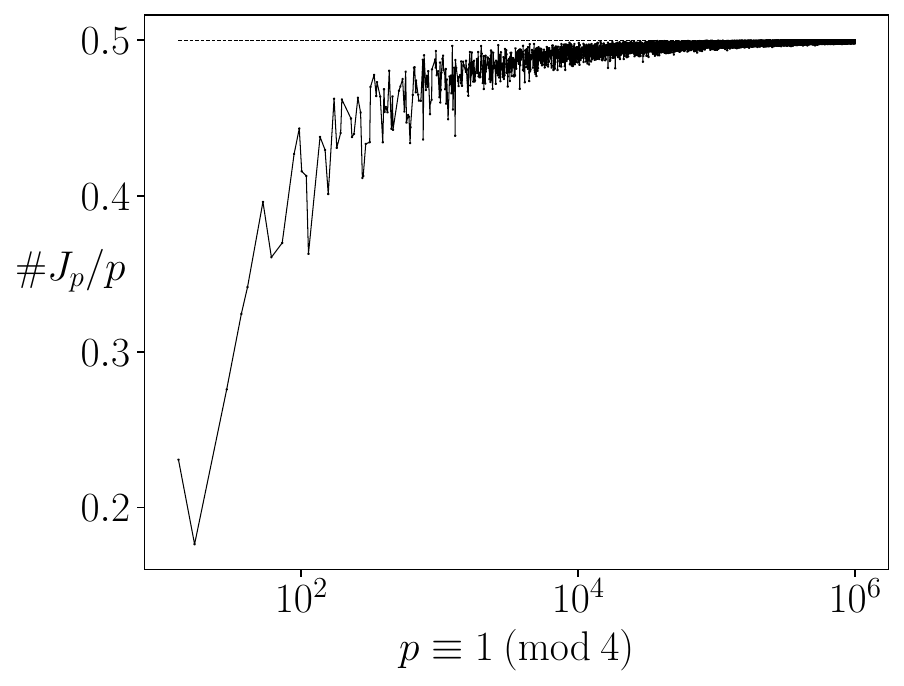}
\end{subfigure}
\caption{Plots of $\#J_p$ and $\#J_p/p$ against varying prime numbers $p\equiv 1\pmod{4}$. In the left panel, the dashed line represents the identity line, where the values along the horizontal and vertical axes are equal to each other. Note that both the axes are in logarithmic scale. In the right panel, the horizontal dashed line represents the value 0.5 on the vertical axis. Note that the horizontal axis is in logarithmic scale.}
\label{fig:J_p}
\end{figure}

In order to investigate the behavior of $\#J_p$ for larger values of $p$, we computed the exact values of $\#J_p$ for $p$ up to $10^6$ based on the algorithm described in \cref{subsec:algorithms}.
The full data of computed values of $\#J_p$, $l_{\mathrm{L}}$, and $l_{\mathrm{R}}$ for $p\equiv 1\pmod{4}$ with $13\leq p\leq 10^6$ is available at \cite{Kobayashi_github}.
In the left panel of \cref{fig:J_p}, we can clearly see that $\#J_p$ steadily increases with $p$.

Since both the axes of the plot are in logarithmic scale in the left panel, the curve of $\#J_p$ against $p$ for larger values of $p$, which is nearly parallel to the identity line, indicates that $\#J_p$ is approximately proportional to $p$.
By calculating and plotting $\#J_p/p$ against $p$, we can see the overall level of the constant of proportionality for this relationship.
Consequently, we obtained the right panel of \cref{fig:J_p}, which visualizes the tendency of $\#J_p/p$ approaching 0.5 --- the highest possible supremum --- as $p$ increases, and thus supports \cref{conj:Num_Jp_over_p_half}.

\subsection{On algorithms}\label{subsec:algorithms}
In \cref{alg:gobel_integrality}, we present the pseudocode for calculating the value of $N_k$ for given values of $k$.
The main part of our algorithm (lines 31--42) works by iteratively determining whether $g_{k,2}(n)\in\mathbb{Z}$ or not under the assumption that $g_{k,2}(i)\in\mathbb{Z}$ for all $1\leq i<n$.
In order to determine the exact value of $N_k$, it is sufficient to start this procedure from $n=2$ (line 32), as $g_{k,2}(1)=2\in\mathbb{Z}$ holds by definition.

For each iteration of $n_\mathrm{max}$, the \texttt{CUMULATIVE\_PRODUCT} function is used at first to obtain the appropriate initial modulus $P$ (see lines 1--18) with $P = \prod_{p\mid n} p^{\nu_p(n!)}$, where $\nu_p$ denotes the $p$-adic valuation for a prime number $p$.
Once the appropriate initial modulus $P$ for $g_{k,2}(1)$ is obtained, we can determine the residue class modulo $n$ that $ng_{k,2}(n)\in \Z$ belongs to.
This is carried out step-by-step using the algorithm outlined as the \texttt{GOBEL\_PROCEED} function (lines 20--29).
Note that the modulus decreases as the iteration proceeds due to the $(n+1)^{-1}$ term in the recurrence relation.
In this iterative process, the integrality of $g_{k,2}(i)$ for $i$ up to $n-1$ can be assumed, because the sequence has already been confirmed to be integral up to the $(n-1)$-st term when the for-loop of lines 32--41 is executed for $n_\mathrm{max} = n$.
If the variable $g_\textrm{mult}$, which is a representative of $(n+1)g_{k,2}(n+1)$ modulo $d$, is not divisible by the greatest common divisor of $d$ and $n+1$ (line 24), the function returns a null result to indicate that the integrality breaks at $n+1$.

The implementation of the function \texttt{powmod} is critical to the efficiency of \cref{alg:gobel_integrality}.
Note that this function is called $O((N_k)^2)$ times as $n_\mathrm{max}$ increments in the for-loop of lines 32--41.
In the previous implementations of calculation of $N_{k,l}$ values \cites{MatsuhiraMatsusakaTsuchida2024,GimaEtal2024+}, the modulo operation was performed \emph{after} the exponentiation.
However, because we only need the residue class modulo $d$, the explicit exponentiation is unnecessary.
Instead, we used the built-in \texttt{pow} function in Python (version 3.8 or later), which employs left-to-right binary exponentiation and left-to-right $k$-ary sliding window exponentiation algorithms, depending on whether the bit length of the power exponent exceeds 60.
The \texttt{invmod} function is also readily available as the \texttt{pow} function in Python (version 3.8 or later) by setting the exponent argument to a negative integer and thus employed in our implementation.
For \texttt{factorint} and \texttt{gcd}, we used the \texttt{ntheory.factorint} function from the \texttt{SymPy} package (version 1.12) and the \texttt{gcd} function of the built-in \texttt{math} module in Python, respectively.
For future coding, we also note that the implementation of \texttt{CUMULATIVE\_PRODUCT} using the big integer type in Python could be more efficient if we instead used the dictionary that represents the prime factorization as an object that represents an integer.
\begin{algorithm}
\caption{Pseudocode for computing $N_k$ for given values of $k$.}\label{alg:gobel_integrality}
\begin{algorithmic}[1]
\Require \texttt{factorint}($i$) is a dictionary \{\textrm{prime}:~exponent\}, the prime factorization of $i$
\Require \texttt{powmod}($x, y, m$) is the modular exponentiation $x^y\:\textrm{mod}\:m$ for $x\in\mathbb{Z}, y,m\in\mathbb{Z}_{>0}$
\Require \texttt{invmod}($x, m$) is the modular multiplicative inverse $\textrm{mod}\:m$ of $x\in\mathbb{Z}$
\Require \texttt{gcd}($x, y$) is the greatest common divisor of integers $x$ and $y$
\vspace{0.7\baselineskip}

\State $L_f, L_c \gets$ List of (pre-determined) length ${i_\mathrm{max}}$ 
\State $D \gets$ Empty dictionary
\For{$i \gets 1$ to $i_\mathrm{max}$}
    \State $L_f[i] \gets \texttt{factorint}(i)$
    \State Add to $D$ every key in $L_f[i]$ that is not in $D$, with the corresponding value 0
    \For{every key $p$ of $L_f[i]$}
        \State $D[p] \gets D[p] + L_f[i][p]$
    \EndFor
    \State $L_c[i] \gets D$
\EndFor
\State {}

\Function{\texttt{CUMULATIVE\_PRODUCT}}{$n, L_f, L_c$}
    \State $P \gets 1$
    \For{every $p$ in keys of $L_f[n]$}
        \State $P \gets P*p^{L_c[n][p]}$
    \EndFor
    \State \algorithmicreturn~$P$
\EndFunction
\State {}

\Function{\texttt{GOBEL\_PROCEED}}{$g, d, n, k$}
    \State $g_\textrm{mult} \gets ng + \texttt{powmod}(g, k, d)$
    \State $m_\textrm{gcd} \gets \texttt{gcd}(d, n + 1)$
    \State $d_\textrm{next} \gets d/m_\textrm{gcd}$
    \If{$g_\textrm{mult}$ is not divisible by $m_\textrm{gcd}$}
        \State \algorithmicreturn~Null
    \EndIf
    \State $g_\textrm{next} \gets \texttt{invmod}((n + 1)/m_\textrm{gcd}, d_\textrm{next})*(g_\textrm{mult}/m_\textrm{gcd})$
    \State \algorithmicreturn~$(g_\textrm{next}\:\mathrm{mod}\: d_\textrm{next}, d_\textrm{next})$
\EndFunction
\State {}

\For{every $k$ for which $N_k$ is to be determined} \Comment{Main part of the algorithm}
\For{$n_\mathrm{max} \gets 2$ to the length of $L_f$}
    \State $(g, d) \gets (2, \texttt{CUMULATIVE\_PRODUCT}(n_\mathrm{max}, L_f, L_c))$
    \For{$n \gets 1$ to ($n_\mathrm{max} - 1)$}
        \State $(g, d) \gets \texttt{GOBEL\_PROCEED}(g, d, n, k)$
    \EndFor
    \If{$(g, d)$ is Null}
        \State Print $(k, n_\mathrm{max})$ \Comment{The value of $n_\mathrm{max}$ here should be $N_k$}
        \State \textbf{break}
    \EndIf
\EndFor
\EndFor
\end{algorithmic}
\end{algorithm}

If we are interested only in the non-integrality of $(k,l)$-G\"obel sequences and not in the exact value of $N_{k,l}$, we can perform the numerical search more efficiently.
We first describe the case $l=2$.
Recall that $N_k = 19$ for all $k\equiv 6,14\pmod{18}$, and the proof consists of two parts: first, the integrality of $k$-G\"obel sequences for any $k\geq 2$ does not break for $n < 19$ and second, $19g_{k,2}(19)\not\equiv 0\pmod{19}$ for $k\equiv 6,14 \pmod{18}$. Omitting the first part of the proof, we still obtain $N_k \leq 19$ for all $k\equiv 6,14\pmod{18}$.
In general cases, the appropriate initial modulus $P$ can become very large for composite $n$. However, because $P=n$ for prime $n$, it is comparatively easy to complete all the possible cases modulo $P$.
Moreover, the computational result that $N_k \leq p$ for $k = a$ extends to all $k\equiv a\pmod{p-1}$ by Fermat's little theorem.
Therefore, for a given range of integers $k$, we can progressively sieve out $k\equiv a\pmod{p-1}$ for $a$ such that $pg_{a,2}(p)\not\equiv 0\pmod{p}$.
The algorithm can also be made more efficient for prime $n$, as it is guaranteed that $m_\textrm{gcd} = 1$ for $n < n_\mathrm{max}$ at line 22 of \cref{alg:gobel_integrality}.

In the actual computation, the residue classes to be sieved out from $k$ were precomputed for prime $p < 30000$.
We first sieved $k\leq 10^{14}$ by the residue classes modulo ($p-1$) for prime $p < 15000$, which left 132603 integers greater than 1.
This calculation took approximately 22 days with a single NVIDIA RTX A6000 GPU and the \texttt{Tensor} class from \texttt{PyTorch} package.
These resulting integers were then subjected to sequential sieving to obtain 2810 and 56 integers that remained after sieving by the residue classes modulo ($p-1$) for prime $p<20000$ and $p<25000$, respectively.
The last integer, 26626531900321, was sieved out by 26059 modulo 29362.
The full data and actual codes are available at \cite{Kobayashi_github}.

We took almost the same approach to study the $(k,l)$-G\"obel sequences computationally.
In this case, it is sufficient to replace the initial value 2 for $g$ with the parameter $l$ at line 33 of \cref{alg:gobel_integrality}. We used this modified algorithm to obtain the values of $(k,l)$ such that $pg_{k,l}(p)\not\equiv 0\pmod{p}$ for prime $p$ up to 3000. The result was used to produce \cref{fig:blue_dots}.

Lastly, we shortly describe the algorithm for computing $l_{\mathrm{L}}$ and $l_{\mathrm{R}}$.
According to \cref{thm:main3},
if there exists an $l \in J_p$,
then $l_\mathrm{L}\leq l<l_\mathrm{R}$.
Furthermore, $0 \in I_p^\mathrm{L}$ and $p-1\in I_p^\mathrm{R}$ always hold.
Therefore, we can apply the binary search to determine $l_\mathrm{L}$ and $l_\mathrm{R}$ efficiently once we found an $l \in J_p$.
We computed the actual values of $l_\mathrm{L}$ and $l_\mathrm{R}$ for $p\equiv 1\pmod{4}$ up to $10^6$. 
The cardinality of $J_p$ was calculated by $\#J_p = (l_\mathrm{R} - l_\mathrm{L})/2$.


\end{document}